\documentclass[letterpaper,12pt]{article}
\usepackage{amsmath,amssymb,amsfonts,epsfig}

\newtheorem{theo}{Theorem}[section]

\newtheorem{lemma}[theo]{Lemma}
\newtheorem{defn}[theo]{Definition}

\newtheorem{cor}[theo]{Corollary}
\newtheorem{example}[theo]{Example}

\newtheorem{quest}[theo]{Question}

\newtheorem{problem}[theo]{Problem}

\newenvironment{proof}{\noindent {\sc Proof}.}
                {\phantom{a} \hfill \framebox[2.2mm]{ } \bigskip}

\newcommand{\NN}{\mathbb{N}}
\newcommand{\ZZ}{\mathbb{Z}}

\def\int{{\rm int}}

\def\cc{{\rm c}}
\newcommand{\G}{{\cal G}}

\renewcommand{\L}{{\cal L}}
\newcommand{\D}{{\mathcal{D}}}
\renewcommand{\b}{\backslash}
\newcommand{\eps}{\varepsilon}
\renewcommand{\phi}{\varphi}
\renewcommand{\theta}{\vartheta}

\title{Eulerian properties of hypergraphs}
\author{M. Amin Bahmanian\footnote{Email: mbahman@ilstu.edu. Mailing address: Department of Mathematics, Illinois State University, Stevenson Hall 313, Campus Box 4520, Normal, Illinois, 61790-4520, USA.} $\,$
and Mateja \v{S}ajna\footnote{Email: msajna@uottawa.ca.  Mailing address: Department of Mathematics and Statistics, University of Ottawa, 585 King Edward Avenue, Ottawa, ON, K1N 6N5,Canada.} }

\begin{document}
\maketitle \baselineskip 17pt

\begin{abstract}
In this paper we study three substructures in hypergraphs that generalize the notion of an Euler tour in a graph. A {\em flag-traversing tour} of a hypergraph corresponds to an Euler tour of its incidence graph, hence complete characterization of hypergraphs with an Euler tour follows from Euler's Theorem. An {\em Euler tour} is a closed walk that traverses each edge of the hypergraph exactly once; and an {\em Euler family} is a family of closed walks that cannot be concatenated and that jointly traverse each edge of the hypergraph exactly once. Lonc and Naroski have shown that the problem of existence of an Euler tour is NP-complete even on a very restricted subclass of 3-uniform hypergraphs, while we show that the problem of existence of an Euler family is polynomial on the class of all hypergraphs.

Furthermore, we examine the necessary conditions for a hypergraph to admit an Euler family (Euler tour, respectively); we show that while these necessary conditions are sufficient for connected graphs, they are  not sufficient for general hypergraphs. On the other hand, we exhibit a new class of hypergraphs for which these necessary conditions are also sufficient, extending a result by Lonc and Naroski. We give a partial characterization of hypergraphs with an Euler family (Euler tour, respectively) in terms of the intersection graph of the hypergraph, and a complete (but not easy to verify) characterization in terms of the incidence graph.   For hypergraphs with an Euler family, we give a complete verifiable characterization using a theorem of Lov\'{a}sz, and then show that every 3-uniform hypergraph without cut edges admits an Euler family. Finally,  we show that a hypergraph admits an Euler family if and only if it can be decomposed into cycles, and exhibit a relationship between 2-factors in a hypergraph and eulerian properties of its dual.

\end{abstract}

\newpage

\tableofcontents

\section{Introduction}

As first claimed by Euler in 1741 \cite{Eul}, and proved by Hierholzer and Wiener in 1873 \cite{HieWie}, it is now well known that a connected graph admits an Euler tour --- that is, a closed walk traversing each edge exactly once --- if and only if it has no vertices of odd degree. In this paper, we are concerned with the analogous problem for hypergraphs. As we shall see, this extended problem is much more complex; in fact, there is more than one natural way to generalize the notion of an Euler tour to hypergraphs. In this paper we shall consider three such natural extensions.

To our knowledge, not much has been previously known about eulerian properties of hypergraphs. The most in-depth treatment to date can be found in \cite{LonNar}, where  Euler tours (closed walks traversing each edge exactly once) of $k$-uniform hypergraphs are considered. In particular, the authors of \cite{LonNar} determine some necessary conditions (see our Lemma~\ref{lem:NecCond}) for existence of an Euler tour in a hypergraph, and show that these are also sufficient for certain classes of $k$-uniform hypergraphs (Theorem~\ref{the:LonNar}). They also show that the problem {\sc Euler Tour} is NP-complete on the class of $k$-uniform hypergraphs for each $k \ge 3$, and even just on the class of  3-uniform hypergraphs with a connected skeleton; see \cite[Theorem 7]{LonNar}.

Some results on Euler tours in block designs have been previously obtained under the disguise of universal cycles and rank-2 universal cycles \cite{Dew}, as well as 1-overlap cycles \cite{HorHur, HorHur4}. In particular, these results imply existence of Steiner triple systems \cite[Theorem 22]{HorHur} and Steiner quadruple systems \cite[Theorem 1.2]{HorHur4} with an Euler tour for all admissible orders greater than 4, existence of twofold triple systems with an Euler tour for many congruency classes of admissible orders \cite[Theorem 5.10]{Dew}, as well as existence of an Euler tour in every cyclic twofold triple system \cite[Corolllary 5.11]{Dew} and every cyclic Steiner triple system of order greater than 3 \cite[Corolllary 5.10]{Dew}.

As we shall see in Subsection~\ref{sec:L(H)}, existence of an Euler tour in a hypergraph implies existence of a Hamilton cycle in its intersection graph, but not conversely. A lot of work has been done on hamiltonicity of block-intersection graphs of designs. The most comprehensive of these results, due to Alspach et al. \cite{AlsHeiMoh}, shows that the block-intersection graph of every pairwise balanced design with index 1 is hamiltonian; however, this construction does not give rise to an Euler tour in the design. Finally, we mention the article \cite{SamPas}, where certain closed walks in hypergraphs (called s-eulerian, p-eulerian, 2-eulerian, and f-eulerian) are studied; however, these concepts are only remotely related to eulerian properties of hypergraphs considered presently.

In this paper, we investigate three natural generalizations of the notion of Euler tour to hypergraphs: (1) a {\em flag-traversing tour} of a hypergraph, which corresponds to an Euler tour of the incidence graph; (2) an {\em Euler tour}, a closed walk that traverses each edge of the hypergraph exactly once; and (3) an {\em Euler family}, a family of closed walks that cannot be concatenated and that jointly traverse each edge of the hypergraph exactly once. We show that the second of these problems is NP-complete even on the very restricted class of linear 2-regular 3-uniform hypergraphs, while the other two problems are polynomial on the class of all hypergraphs. In fact, Euler's Theorem \cite{Eul} implies a complete and easy characterization of hypergraphs with flag-traversing tours, and this result will be presented in Subsection~\ref{sec:ET}. The rest of the paper is devoted to  Euler tours and Euler families. As expected for an NP-complete problem, only partial results will be offered for the problem of existence of  Euler tours in hypergraphs;  we shall present some necessary conditions, some sufficient conditions, complete characterization for certain classes of hypergraphs, as well as characterization in terms of the intersection graph, the incidence graph, and the blocks of the hypergraph. Analogous, as well as additional results will be presented for hypergraphs with Euler families.

The paper is organized as follows. In the remainder of this section, we introduce basic hypergraph terminology and the three eulerian substructures in hypergraphs, as well as completely characterize hypergraphs with a flag-traversing Euler tour. In the second part of the paper (Section~\ref{sec:main}), we focus on hypergraphs with Euler tours and Euler families. First, we examine the necessary conditions for a hypergraph to admit an Euler tour or Euler family; we show that while these necessary conditions are sufficient for connected graphs, they are  not sufficient for general hypergraphs. On the other hand, we exhibit a new class of hypergraphs for which these necessary conditions are also sufficient. In Subsection~\ref{sec:L(H)} we then give a partial characterization in terms of the intersection graph of the hypergraph, and in Subsection~\ref{sec:G}, a complete (but not easy to verify) characterization in terms of the incidence graph. Block structure with respect to eulerian properties of a hypergraph is considered in Subsection~\ref{sec:blocks}. In Subsection~\ref{sec:compl}, we determine the complexity of the problems {\sc Euler Tour} and {\sc Euler Family}, and in the next two sections we focus on the latter. We give a complete verifiable characterization of hypergraphs with an Euler family using a theorem of Lov\'{a}sz, and then show that every 3-uniform hypergraph without cut edges admits an Euler family. Finally, in Subsection~\ref{sec:CD}, we show that a hypergraph admits an Euler family if and only if it can be decomposed into cycles, and exhibit a relationship between 2-factors in a hypergraph and eulerian properties of its dual.

\subsection{Preliminaries}

For any graph-theoretic terms not defined here, the reader is referred to \cite{BonMur}, and for hypergraph definitions, to our earlier manuscript \cite{BahSaj}. Note that, since 2-uniform hypergraphs can be thought of as loopless graphs, most terms defined below also extend to graphs.

A hypergraph $H$ is an ordered pair $(V,E)$, where  $V$ and $E$ are disjoint finite sets such that $V \ne \emptyset$, together with a function  $\psi: E \rightarrow 2^V$, called the {\em incidence function}. The elements of $V=V(H)$ are called {\em vertices}, and the elements of $E=E(H)$ are called {\em edges}. The number of vertices $|V|$ and number of edges $|E|$ are called the {\em order} and {\em size} of the hypergraph, respectively. Often we denote $n=|V|$ and $m=|E|$. A hypergraph with a single vertex is called {\em trivial}, and a hypergraph with no edges is called {\em empty}.

Two edges $e,e' \in E$ are said to be {\em parallel} if $\psi(e)=\psi(e')$, and the number of edges parallel to edge $e$ (including $e$) is called the {\em multiplicity} of $e$. A hypergraph $H$ is called {\em simple} if no edge has multiplicity greater than 1; that is, if $\psi$ is injective.

\bigskip

As is customary for graphs, the incidence function may be omitted when no ambiguity can arise (in particular, when the hypergraph is simple, or when we do not need to distinguish between distinct parallel edges). An edge $e$ is then identified with the subset $\psi(e)$ of $V$, and for $v \in V$ and $e \in E$, we then more conveniently write $v \in e$ or $v \not\in e$ instead of $v \in \psi(e)$ or $v \not\in \psi(e)$, respectively.

\bigskip

Let $H=(V,E)$ be a hypergraph. If $v,w \in V$ are distinct vertices and there exists $e \in E$ such that $v,w \in e$, then $v$ and  $w$ are said to be {\em adjacent} in $H$ (via edge $e$). Similarly, if $e,f \in E$ are distinct (but possibly parallel) edges and $v \in V$ is  such that $v \in e \cap f$, then $e$ and $f$ are said to be {\em adjacent} in $H$ (via vertex $v$).

Each ordered pair $(v,e)$ such that $v \in V$, $e \in E$, and $v \in e$ is called a {\em flag} of $H$; the set of flags is denoted by $F(H)$. If $(v,e)$ is a flag of $H$, then we say that vertex $v$ is {\em incident} with edge $e$.

The {\em degree} of a vertex $v \in V$ (denoted by $\deg_H(v)$ or simply $\deg(v)$ if no ambiguity can arise) is the number of edges $e \in E$ such that $v \in e$. A vertex of degree 0 is called {\em isolated}, and a vertex of degree 1 is called {\em pendant}. A hypergraph $H$ is said to be {\em $r$-regular on $V'$}, for $V' \subseteq V$, if every vertex in $V'$ has degree $r$ in $H$, and simply {\em $r$-regular} if it is $r$-regular on $V$. Similarly, $H$ is said to be {\em even (odd) on $V'$} if every vertex of $V'$ has even (respectively, odd) degree in $H$, and simply {\em even (odd)} if it is even (respectively, odd) on $V$.

The maximum (minimum) cardinality $|e|$ of any edge $e \in E$ is called the {\em rank} ({\em corank}, respectively) of $H$. A hypergraph $H$ is {\em uniform of rank} $r$ (or {\em $r$-uniform}) if $|e|=r$ for all $e \in E$. An edge $e \in E$ is called 
{\em empty} if $|e|=0$.

\bigskip

A hypergraph $H'=(V',E')$ is called a {\em hypersubgraph} of a hypergraph $H=(V,E)$ if $V' \subseteq V$ and $E' \subseteq E$. For $E' \subseteq E$, the hypergraph $(\cup_{e \in E'} e,E')$, denoted by $H[E']$, is called the {\em hypersubgraph of $H$ induced by the edge set $E'$}.
If $e \in E$, we write shortly $H-e$ for the hypersubgraph $(V,E-\{e \})$, also called an {\em edge-deleted hypersubgraph}.
A hypersubgraph $H'=(V',E')$ of $H$ is called {\em spanning} if $V'=V$.
An {\em $r$-factor} of $H$ is a spanning $r$-regular hypersubgraph of $H$.
Furthermore, if $a$ and $b$ are integers with $0 \le a \le b$, then we define an {\em $(a,b)$-factor} of $H$ as a spanning hypersubgraph of $H$ in which every vertex has degree in the interval $[a,b]$.

\bigskip

If $H=(V,E)$ is a hypergraph  and $V' \subseteq  V$, then $H[V']$, called the {\em subhypergraph of $H$ induced by $V'$}, is the hypergraph obtained from $H$ by deleting all vertices of $V-V'$ from $V$ and from every edge of $H$, and subsequently deleting all empty edges. (See \cite{BahSaj} for a definition of general subhypergraphs.) For $v \in V$, the hypergraph $H[V-\{ v \}]$ is also denoted by $H \b v$ and  called a {\em vertex-deleted subhypergraph} of $H$.

\bigskip

A hypergraph is called {\em linear} if every pair of distinct edges intersect in at most one vertex.

\bigskip

The {\em union} of hypergraphs is a straight generalization of union of graphs.
If the edge set of a hypergraph $H$ is a disjoint union of the edge sets of its hypersubgraphs $H_1,\ldots,H_k$, then we say that $H$ {\em decomposes} into $H_1,\ldots,H_k$, and write $H=H_1 \oplus \ldots \oplus H_k$. A decomposition of a hypergraph $H$ into its $r$-factors is called an {\em $r$-factorization} of $H$.

\bigskip

The main tool used in this paper is the conversion of a problem about hypergraphs to a problem about graphs. The incidence graph, to be defined below, is particularly helpful since it contains complete information about its hypergraph.

Let  $H=(V,E)$ be a hypergraph with incidence function $\psi$. The {\em incidence graph} $\G(H)$ of $H$ is the graph $\G(H)=(V_{G},E_G)$ with $V_G=V \cup E$ and $E_G=\{ ve: v\in V, e \in E, v \in \psi(e)\}$.
Thus $\G(H)$ is a bipartite simple graph with bipartition $\{ V, E \}$. We call a vertex $x$ of $\G(H)$ a {\em v-vertex} if $x \in V$, and an {\em e-vertex} if $x \in E$. Note that the edge set of $\G(H)$ can be identified with the flag set $F(H)$; that is, $E_G=\{ ve: (v,e) \in F(H)\}$.

The {\em intersection graph} (or {\em line graph}) of the hypergraph $H=(V,E)$, denoted $\L(H)$, is the graph with vertex set $E$ and edge set $\{ e e': e, e'\in E, e \ne e', e \cap e' \ne \emptyset \}$. More generally, for any positive integer $\ell$, we define the {\em $\ell$-intersection graph} of the hypergraph $H=(V,E)$, denoted $\L_{\ell}(H)$, as the graph with vertex set $E$ and edge set $\{ e e': e, e'\in E, e \ne e', |e \cap e'| = \ell \}$, and the {\em $\ell^*$-intersection graph} of $H$, denoted $\L_{\ell}^*(H)$, as the graph with vertex set $E$ and edge set $\{ e e': e, e'\in E, e \ne e', |e \cap e'| \ge \ell \}$. Note that the intersection graphs do not contain full information about the hypergraph.

\bigskip

Let $H=(V,E)$ be a hypergraph, let $u,v \in V$, and let $k \ge 0$ be an integer. A {\em $(u,v)$-walk of length $k$}  in $H$ is a sequence $v_0 e_1 v_1 e_2 v_2 \ldots v_{k-1} e_k v_k$ of vertices and edges (possibly repeated) such that $v_0,v_1,\ldots,v_k \in V$, $e_1,\ldots,e_k \in E$, $v_0=u$, $v_k=v$, and for all $i=1,2,\ldots, k$, the vertices $v_{i-1}$ and $v_i$ are adjacent in $H$ via the edge $e_i$. Vertices $v_0,v_1,\ldots,v_k$ are called the {\em anchors} of $W$; $v_0$ and $v_k$ are the {\em endpoints}, and $v_1,\ldots,v_{k-1}$ are the {\em internal vertices}.
Observe that since adjacent vertices are by definition distinct, no two consecutive vertices in a walk are the same. Concatenation of walks is defined in the usual way.

A walk $W=v_0 e_1 v_1 e_2 v_2 \ldots v_{k-1} e_k v_k$  in a hypergraph $H=(V,E)$  is called (i) a {\em trail}, if the anchor flags $(v_0,e_1),(v_1, e_1),(v_1,e_2),\ldots,(v_{k-1},e_k),(v_k,e_k)$ are pairwise distinct; (ii) a {\em strict trail} if the edges $e_1,\ldots,e_k$ are pairwise distinct; and (iii) a {\em path} if both the vertices  $v_0,v_1,\ldots,v_k$ and the edges $e_1,\ldots,e_k$ are pairwise distinct. (Here,  ``distinct'' should be understood in the strict sense; that is, parallel edges need not be distinct.)

A walk  $W=v_0 e_1 v_1 e_2 v_2 \ldots v_{k-1} e_k v_k$ in a hypergraph $H=(V,E)$ is called {\em closed} if $k \ge 2$ and $v_0=v_k$. A {\em closed trail} and {\em closed strict trail} are defined analogously. If  the vertices  $v_0,v_1,\ldots,v_{k-1}$ and the edges $e_1,\ldots,e_k$ are pairwise distinct, then the closed walk $W$ is called a {\em cycle} (sometimes called a {\em Berge cycle} in the literature).

The following result describes the correspondence between the various types of walks in a hypergraph and its incidence graph.

\begin{lemma}\cite{BahSaj}\label{lem:W-W_G}
Let $H=(V,E)$  be a hypergraph and $G=\G(H)$ its incidence graph.  Let $v_i \in V$ for $i=0,1,\ldots,k$, and $e_i \in E$ for $i=1,\ldots,k$, and let $W=v_0 e_1 v_1 e_2 v_2 \ldots v_{k-1} e_k v_k$. Denote the corresponding sequence of vertices in $G$ by $W_G$. Then the following hold:
\begin{enumerate}
\item $W$ is a (closed) walk in $H$ if and only if $W_G$ is a (closed) walk in $G$ with no two consecutive v-vertices being the same.
\item $W$ is a trail (path, cycle) in $H$ if and only if $W_G$ is a trail (path, cycle, respectively) in $G$.
\item $W$ is a strict trail in $H$ if and only if $W_G$ is a trail in $G$ that visits every $e \in E$ at most once.
\end{enumerate}
\end{lemma}

A hypergraph $H=(V,E)$ is said to be {\em connected} if for every pair of distinct vertices $u,v \in V$ there exists a $(u,v)$-walk (or equivalently, a $(u,v)$-path) in $H$. The {\em connected components} of $H$ are the maximal connected hypersubgraphs of $H$ that have no empty edges. The number of connected components of $H$ is denoted by $\cc(H)$.

\begin{theo}\cite{BahSaj}\label{the:conn}
Let $H=(V,E)$ be a hypergraph  without empty edges. Then $H$ is connected if and only if its incidence graph $G=\G(H)$ is connected.
\end{theo}


\subsection{Introduction to eulerian properties of hypergraphs}

An Euler tour in a graph is usually defined as a closed trail that traverses every edge of the graph. Equivalently, an Euler tour in a graph is a closed trail that traverses each flag exactly once. This observation suggests two natural ways to generalize Euler tours to hypergraphs. We add a third one, observing that in a connected graph, a family of closed trails that jointly traverse each edge exactly once, can always be concatenated into an Euler tour.

\begin{defn}{\rm
Let  $H=(V,E)$ be a hypergraph.
\begin{enumerate}
\item A {\em flag-traversing tour} of $H$ is a closed trail of $H$ traversing every flag of $H$.

\item An {\em  Euler tour} of $H$ is a closed strict trail of $H$ traversing every edge of $H$.

\item An {\em Euler family} of $H$ is a family ${\cal F}=\{ T_1,\ldots,T_k\}$ of closed strict  trails of $H$ such that: \\ (i) each edge of $H$ lies in exactly one trail of the family, and \\ (ii) the trails $T_1,\ldots,T_k$ are
pairwise anchor-disjoint.
\end{enumerate}}
\end{defn}

We remark that  a hypergraph admits a family of closed strict trails satisfying Property (i) if and only if it admits a family satisfying Properties (i) and (ii), since two closed strict trails with a common anchor can be concatenated into a longer closed  strict trail.


The main objective of this paper is to characterize hypergraphs with a flag-traversing tour,  Euler tour,  and Euler family, respectively. As we shall see, these three seemingly similar problems --- which are, in fact, equivalent for connected graphs --- greatly differ in their difficulty. In the next section, we completely solve the first problem. Then, in the remainder of the paper, we give partial solutions to the latter two problems. In particular, we show that the second problem is NP-complete even on a very restricted subclass of hypergraphs, while the third is polynomial on the set of all hypergraphs.


\subsection{Hypergraphs with a flag-traversing tour}\label{sec:ET}

\begin{theo}\label{the:Etour}
A connected hypergraph $H=(V,E)$ has a flag-traversing tour if and only if its incidence graph  has an Euler tour, that is, if and only if $\deg_H(v)$ and $|e|$ are even for all $v \in V$ and $e \in E$.
\end{theo}

\begin{proof}
Let $H=(V,E)$ be a connected hypergraph, and $G$ its incidence graph. A flag-traversing of $H$ is a closed walk $W$ of $H$ that traverses each flag $(v,e)$ of $H$ exactly once. Hence, by Lemma~\ref{lem:W-W_G}, $W$ corresponds to a closed walk of $G$ that traverses each edge of $G$ exactly once, that is, an Euler tour of $G$; and vice-versa. The result follows by Euler's Theorem \cite{Eul}.
\end{proof}

\begin{cor}
Let $H=(V,E)$ be a hypergraph such that $\deg_H(v)$ and $|e|$ are even for all $v \in V$ and $e \in E$. Then $H$ is admits a collection of cycles such that each flag of $H$ is an anchor flag of exactly one of these cycles.
\end{cor}

\begin{proof}
With the assumptions of the corollary, the incidence graph $G=\G(H)$ is even, and hence is by Veblen's Theorem \cite{Veb} an edge-disjoint union of cycles. Each cycle $C_G$ of $G$ corresponds to a cycle $C_H$ in $H$ by Lemma~\ref{lem:W-W_G}, and the edges of $C_G$ correspond to the anchor flags of $C_H$. Hence $H$ admits a collection of cycles such that each flag of $H$ is an anchor flag of exactly one of them.
\end{proof}

Note that the above corollary does not claim that $H$ admits a decomposition into cycles; compare Theorem~\ref{the:CD}.

Since hypergraphs with a flag-traversing tour are completely characterized in Theorem~\ref{the:Etour}, we shall focus on  Euler tours and Euler families for the rest of the paper. We hence define the following terms.

\begin{defn}{\rm
A hypergraph is called {\em eulerian} if it admits an Euler tour, and {\em quasi-eulerian} if it admits an Euler family. }
\end{defn}

Clearly, every eulerian hypergraph is also quasi-eulerian, but as we shall see soon, the converse does not hold.


\section{Eulerian and quasi-eulerian hypergraphs}\label{sec:main}

All hypergraphs in this section are assumed to have no empty edges.


\subsection{Examining the necessary conditions}\label{sec:NC}

Lonc and Naroski \cite{LonNar} determined the following necessary conditions for a ($k$-uniform) hypergraph to be eulerian. We extend their observation to general quasi-eulerian hypergraphs, and since the two conditions are equivalent, we shall refer to them in the singular.

\begin{lemma}\label{lem:NecCond}
Let $H=(V,E)$ be a quasi-eulerian hypergraph, and $V_{odd}$ the set of odd-degree vertices in $H$. Then
\begin{equation}\label{eq1}
|E| \le \sum_{v \in V} \lfloor \frac{\deg_H(v)}{2} \rfloor
\end{equation}
and
\begin{equation}\label{eq2}
|V_{odd}| \le \sum_{e \in E} (|e|-2).
\end{equation}
Moreover, the two inequalities are equivalent.
\end{lemma}

\begin{proof}
First, we show Inequality~(\ref{eq1}). Let $\cal F$ be an Euler family of $H$, and $T$ any closed strict trail in $\cal F$. For all $u \in V$, let $m_T(u)$ denote the number of times $u$ is traversed on $T$ as an anchor vertex. (Here, the endpoints of the trail together count as one traversal.) Since vertices and edges alternate along $T$, and each edge of $H$ is traversed exactly once by a trail in $\cal F$, we have $\sum_{T \in {\cal F}}\sum_{v \in V} m_T(v) = |E|$. Clearly, for all $v \in V$, we have $\sum_{T \in {\cal F}}m_T(v) \le \lfloor \frac{\deg_H(v)}{2} \rfloor$. Hence, as claimed,
$$\sum_{v \in V} \lfloor \frac{\deg_H(v)}{2} \rfloor \ge \sum_{v \in V} \sum_{T \in {\cal F}} m_T(v) = \sum_{T \in {\cal F}}\sum_{v \in V} m_T(v) =|E| .$$

It now suffices to show that Inequalities~(\ref{eq1}) and (\ref{eq2}) are equivalent. Observe that
\begin{eqnarray*}
\sum_{v \in V} \lfloor \frac{\deg_H(v)}{2} \rfloor  &=& \sum_{v \in V_{odd}} \frac{\deg_H(v)-1}{2} + \sum_{v \in V-V_{odd}} \frac{\deg_H(v)}{2} \\
&=& \frac{1}{2} \left( \sum_{v\in V} \deg_H(v) - |V_{odd}| \right)=\frac{1}{2} \left( \sum_{e\in E} |e| - |V_{odd}| \right).
\end{eqnarray*}
Hence $\sum_{v \in V} \lfloor \frac{\deg_H(v)}{2} \rfloor \ge |E|$ if and only if
$\sum_{e \in E} |e| - |V_{odd}| \ge 2|E|$, that is, if and only if $|V_{odd}| \le \sum_{e \in E} (|e|-2)$.
\end{proof}

We remark that a hypergraph with a single edge cannot be quasi-eulerian, both by the definition of a closed walk, as well as by Lemma~\ref{lem:NecCond}.

\begin{quest}\label{quest1}{\rm
Is the necessary condition in Lemma~\ref{lem:NecCond} also sufficient?
}
\end{quest}

The answer to Question~\ref{quest1} is positive for graphs: if $G=(V,E)$ is a graph, then $\sum_{e \in E} (|e|-2)=0$, and  if $G$ satisfies the condition in Lemma~\ref{lem:NecCond}, then $G$ has no vertices of odd degree, and every connected component of $G$ has an Euler tour. Hence $G$ is quasi-eulerian.

For hypergraphs, however, it is easily seen that in general the answer to the above question is negative. We shall now present some counterexamples, starting with the more trivial ones. We first state the most obvious limiting property, which follows straight from the definition of a walk.

\begin{lemma}\label{lem:CE-0}
A quasi-eulerian hypergraph has no edge of cardinality less than two.
\end{lemma}

\begin{example}{\rm
Fix any $n \ge 3$, and let $V=\{ v_1, v_2,\ldots,v_n \}$ and $E=\{ e_1, e_2,\ldots,e_n \}$, where $e_1=\{ v_1 \}$, $e_2=\{ v_1, v_2, v_n \}$, and $e_i=\{ v_{i-1},v_{i} \}$ for $i=3,\dots,n$.
Then $H=(V,E)$ is 2-regular, whence $\sum_{v \in V} \lfloor \frac{\deg(v)}{2} \rfloor = n=|E|$. Thus $H$ satisfies the necessary conditions in Lemma~\ref{lem:NecCond}. However, $H$ is not quasi-eulerian since $e_1$ is an edge of cardinality 1.
}
\end{example}

\begin{lemma}\label{lem:CE-1}
Every edge of a quasi-eulerian hypergraph contains at least two vertices that are not pendant.
\end{lemma}

\begin{proof}
An anchor vertex of a closed trail (necessarily traversing at least two edges)  must have degree at least 2. If a hypergraph is quasi-eulerian, every edge lies in a closed trail, and hence every edge contains at least 2 vertices that are not pendant.
\end{proof}

\begin{example}{\rm
Fix any $n \ge 7$, and let $V=\{ v_1, v_2,\ldots,v_n \}$ and $E=\{ e_1, e_2,\ldots,e_n \}$, where $e_1=\{ v_1,v_2,v_3 \}$, $e_i=\{ v_{i+1},v_{i+2},v_{i+3} \}$ for $i=2,\dots,n-3$, $e_{n-2}=\{ v_3,v_{n-1},v_{n} \}$, $e_{n-1}=\{ v_4,v_{n-1},v_{n} \}$, $e_{n}=\{ v_5,v_{n-1},v_{n} \}$.
Then $H=(V,E)$ has no isolated vertices, has exactly two vertices of degree 1 (namely, $v_1$ and $v_2$), and at least two vertices of degree at least 4. It follows that  $\sum_{v \in V} \lfloor \frac{\deg(v)}{2} \rfloor \ge n=|E|$. Thus $H$ satisfies the necessary conditions in Lemma~\ref{lem:NecCond}. However, since $e_1$ is an edge with exactly one non-pendant vertex, $H$ is not quasi-eulerian by Lemma~\ref{lem:CE-1}.
}
\end{example}

Observe that if $V'$ is the set of pendant vertices in a hypergraph $H$, then $H$ is quasi-eulerian (eulerian) if and only if $H[V-V']$ is. Hence we shall now look for counterexamples to the sufficiency of the condition in Lemma~\ref{lem:NecCond} among hypergraphs without pendant vertices (and of course, without edges of cardinality less than 2).

\medskip

A {\em cut edge} in a hypergraph $H=(V,E)$ is an edge $e \in E$ such that $\cc(H-e) > \cc(H)$. A graph with a cut edge obviously does not admit an Euler tour. The issue is more complex for hypergraphs.  First, we need to distinguish between two types of cut edges. As we showed in \cite[Lemma 3.15]{BahSaj}, if $e$ is a cut edge in a hypergraph  $H=(V,E)$, then $\cc(H-e) \le \cc(H)+|e|-1$. A cut edge that achieves the upper bound in this inequality is called {\em strong}; all other cut edges are called {\em weak}. Observe that a cut edge has cardinality at least two, and that any cut edge of cardinality two (and hence any cut edge in a graph) is necessarily strong.

\begin{theo}\label{the:cut-edges}
Let $H=(V,E)$ be a hypergraph with a cut edge $e$.
\begin{enumerate}
\item If $e$ is a strong cut edge, then $H$ is not quasi-eulerian.
\item If $H-e$ has at least two non-trivial connected components, then $H$ is not eulerian.
\end{enumerate}
\end{theo}

\begin{proof}
\begin{enumerate}
\item Assume that $e$ is a strong cut edge. If $H$ admits an Euler family, then $e$ lies in a closed strict trail, and consequently in a cycle of $H$. However, by \cite[Theorem 3.18]{BahSaj}, no strong cut edge lies in a cycle ---  a contradiction. Hence $H$ is not quasi-eulerian.
\item Assume that $H-e$ has at least two non-trivial connected components. Since $e$ is a cut edge of $H$, it is a cut e-vertex in its incidence graph $G=\G(H)$ \cite[Theorem 3.23]{BahSaj}, and the connected components of $G\b e$ are the incidence graphs of the connected components of $H-e$ \cite[Lemma 2.8 and Corollary 3.12]{BahSaj}. Hence, by assumption,  $G\b e$ has at least two connected components with e-vertices.

    Suppose $H$ has an Euler tour. By Lemma~\ref{lem:W-W_G}, $G$ has a closed trail $T$ traversing each e-vertex (including $e$) exactly once. Hence $T\b e$ is a trail that traverses every e-vertex in $G\b e$, contradicting the above. Thus $H$ is not eulerian.
\end{enumerate}
\vspace{-10mm}
\end{proof}

\begin{center}
\begin{figure}[t]
\centerline{\includegraphics[scale=0.5]{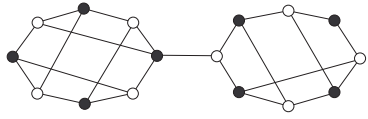}}
\caption{The incidence graph of a 3-uniform quasi-eulerian hypergraph with a cut edge.  (Black dots represent the v-vertices.)}\label{fig:pic4}
\end{figure}
\end{center}

Note that if a hypergraph has no strong cut edges, then it may or may not be quasi-eulerian; an example of each kind is given in Figures~\ref{fig:pic4} and \ref{fig:pic5}, respectively. Both of these examples have weak cut edges and satisfy the necessary condition from Lemma~\ref{lem:NecCond}; they are also 3-uniform.

\begin{center}
\begin{figure}[t]
\centerline{\includegraphics[scale=0.5]{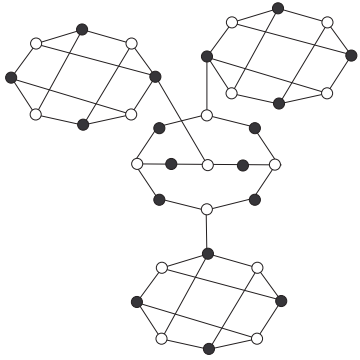}}
\caption{The incidence graph of a 3-uniform hypergraph with cut edges that is not quasi-eulerian.  (Black dots represent the v-vertices.)} \label{fig:pic5}
\end{figure}
\end{center}

As we saw in Theorem~\ref{the:cut-edges}, a cut edge may prevent a hypergraph from admitting an Euler family. We now give a counterexample to the sufficiency of the condition in Lemma~\ref{lem:NecCond} that has no cut edges. It can be easily generalized to give an infinite family of such hypergraphs.

\begin{example}\label{count:2}{\rm
Let  $H=(V,E)$ be a hypergraph whose incidence graph is shown in Figure~\ref{fig:counter2}. Then $H$ has no cut edges and no Euler family, yet it satisfies the necessary condition  from Lemma~\ref{lem:NecCond}.
}
\end{example}

\begin{center}
\begin{figure}[b]
\centerline{\includegraphics[scale=0.5]{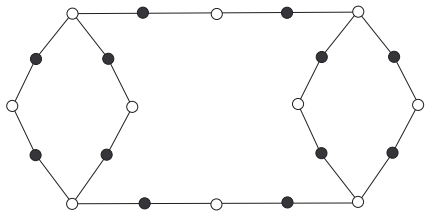}}
\caption{The incidence graph of a hypergraph without cut edges that satisfies the necessary condition from Lemma~\ref{lem:NecCond} but is not quasi-eulerian. (Black dots represent the v-vertices.)}\label{fig:counter2}
\end{figure}
\end{center}

Note that the previous counterexample contains edges of cardinality 2; Stamplecoskie \cite{Sta} recently showed that for every $c \ge 2$ and $m \ge5$,  there exists a connected hypergraph of corank $c$, size $m$, and without cut edges that satisfies the necessary condition from Lemma~\ref{lem:NecCond} but is not quasi-eulerian. However, in Theorem~\ref{the:3-uniform} we shall see that every 3-uniform hypergraph without cut edges is quasi-eulerian. Hence the following  question.

\begin{quest}{\rm
Does there exist a connected $k$-uniform hypergraph (for $k \ge 4$) with no cut edges that satisfies the necessary condition from Lemma~\ref{lem:NecCond} but  is not quasi-eulerian? }
\end{quest}

The following example shows that Theorem~\ref{the:3-uniform} does not extend to eulerian hypergraphs; that is, not all 3-uniform hypergraphs without cut edges are eulerian.

\begin{example}\label{count:3}{\rm
Let  $H=(V,E)$ be a hypergraph whose incidence graph is shown in Figure~\ref{fig:counter3}. Observe that $H$ is 3-uniform and has no cut edges. Furthermore, it is quasi-eulerian but not eulerian, and it satisfies the necessary condition in Lemma~\ref{lem:NecCond}. Observe that its 2-intersection graph is disconnected  --- see Theorem~\ref{the:LonNar} below.
}
\end{example}

\begin{center}
\begin{figure}[t]
\centerline{\includegraphics[scale=0.5]{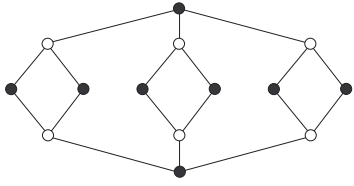}}
\caption{The incidence graph of a 3-uniform hypergraph without cut edges that is quasi-eulerian but not eulerian.  (Black dots represent the v-vertices.)}\label{fig:counter3}
\end{figure}
\end{center}

To conclude this section, we shall present a new class of hypergraphs for which the necessary condition from Lemma~\ref{lem:NecCond} is also sufficient, extending the following result by Lonc and Naroski.

\begin{theo}\cite{LonNar}\label{the:LonNar}
Let $k \ge 3$, and let $H=(V,E)$ be a $k$-uniform hypergraph with a connected $(k-1)$-intersection graph. Then $H$ is eulerian if and only if  $$\sum_{v \in V} \lfloor \frac{\deg_H(v)}{2} \rfloor \ge |E|.$$
\end{theo}

We propose the following two generalizations (Theorem~\ref{the:SuffCond} and Corollary~\ref{cor:SuffCond} below). The main idea of the proof is based on the proof of the above theorem  from \cite{LonNar} for $k\ge 4$.

For any hypergraph $H=(V,E)$, define a digraph $\D_3(H)$ as follows: its vertex set is $E$ and its arc set is $\{ (e,f): e,f \in E, |f-e|=1, |e \cap f|\ge 3 \}$. Recall that an {\em arborescence} is a directed graph whose underlying undirected graph is a tree, and whose arcs are all directed towards a root.

\begin{theo}\label{the:SuffCond}
Let $H=(V,E)$ be a hypergraph such that its digraph $\D_3(H)$ has a spanning subdigraph that is a vertex-disjoint union of non-trivial arborescences. Then $H$ is quasi-eulerian.
\end{theo}

\begin{proof}
For convenience, we say that a digraph satisfies Property P if has a spanning subdigraph that is a vertex-disjoint union of non-trivial arborescences.
We shall prove by induction on the number of edges that every hypergraph $H$ whose  digraph $\D_3(H)$ satisfies Property P possesses an Euler family. Observe that such a hypergraph necessarily has at least two edges.

First, let $H=(V,E)$ be a hypergraph with $E=\{ e,f \}$ such that its  digraph $\D_3(H)$ satisfies Property P. Then $\D_3(H)$ must have a spanning arborescence. Moreover, we have that $|e \cap f| \ge 3$. Take any $u,v \in e \cap f$ such that $ u \ne v$. Then $T=uevfu$ is an Euler tour of $H$. Thus $H$ possesses an Euler family as claimed.

Assume that for some $m\ge 2$, every hypergraph $H$ with at least $m$ edges whose digraph $\D_3(H)$ satisfies Property P possesses an Euler family. Let $H=(V,E)$ be a hypergraph with $|E|=m+1$ such that its  digraph $\D_3(H)$ has a spanning subdigraph $D'$ that is a vertex-disjoint union of non-trivial arborescences. If each arborescence in $D'$ is of order 2, then (just as in the base case above) each gives rise to a closed strict trail of length 2 in $H$, and the union of all these trails is an Euler family in $H$.

Hence assume that $D'$ has a weakly connected component $A$ that is an arborescence of order at least 3. Let $e\in E$ be a leaf (that is, vertex of indegree 0 and outdegree 1) of $A$ and $f$ its outneighbour in $A$. Then $|f-e|=1$ and $|e \cap f| \ge 3$. Now $\D_3(H-e)$ has a spanning digraph $D' \setminus e$ that is a vertex-disjoint union of non-trivial arborescences. Hence by the induction hypothesis, the hypergraph $H-e$ possesses an Euler family $\cal F$. Let $T=ufvW$ ---  where $u$ and $v$ are distinct vertices in $f$, and $W$ is an appropriate $(v,u)$-walk --- be a closed strict trail in $\cal F$. We now reroute $T$ to include the edge $e$, resulting in a closed strict trail $T'$ of $H$, as follows.

Since $|f-e|=1$, at least one of $u$ and $v$ --- say $v$ without loss of generality --- is also in $e$, and since $|e \cap f| \ge 3$, there exists $w \in e \cap f$ such that $w \ne u,v$. Then  $T'=ufwevW$ is a closed strict trail of $H$. Finally, replace $T$ in $\cal F$ by $T'$ to obtain an Euler family of $H$.

The result follows by induction.
\end{proof}

With very minor changes to the above proof we obtain the following.

\begin{cor}\label{cor:SuffCond}
Let $H=(V,E)$ be a hypergraph such that its digraph $\D_3(H)$ has a non-trivial spanning arborescence. Then $H$ is eulerian.
\end{cor}

Observe that Corollary~\ref{cor:SuffCond} extends  Theorem~\ref{the:LonNar} to a (much) larger family of hypergraphs, while Theorem~\ref{the:SuffCond} generalizes it to quasi-eulerian hypergraphs. Still, the sufficient conditions in Theorem~\ref{the:SuffCond} and Corollary~\ref{cor:SuffCond} are very strong, and the converses clearly do not hold.


\subsection{Characterization using the intersection graphs}\label{sec:L(H)}

We shall now take another look at the intersection graphs of a hypergraph $H$ to determine some necessary and some sufficient conditions for $H$ to be eulerian or quasi-eulerian. The following observation will be an essential tool to establish the necessary conditions.

\begin{lemma}\label{lem:line}
Let $H=(V,E)$  be a hypergraph and $L=\L(H)$ its intersection graph. Furthermore, let $W=v_0 e_1 v_1 e_2 v_2 \ldots v_{k-1} e_k v_k$ be a walk in $H$ (so that all $v_i \in V$ and all $e_i \in E$), and let $W_L= e_1 e_2 \ldots e_k$ and $W_L^*= e_1 e_2 \ldots e_k e_1$. Then:
\begin{enumerate}
    \item $W_L$ is a walk in $L$.
    \item If $W$ is a strict trail, then $W_L$ is a path.
    \item If $W$ is a closed strict trail and $k \ge 3$, then $W_L^*$ is a cycle.
\end{enumerate}
\end{lemma}

\begin{proof}
\begin{enumerate}
\item As defined, $W_L$ is a sequence of vertices in $L$ such that any two consecutive vertices are adjacent; that is, $W_L$ is a walk.
\item If $W$ is a strict trail, then none of the edges in $W$ are repeated, and hence none of the vertices in $W_L$ are repeated. Thus $W_L$ is a path.
\item If $W$ is a closed strict trail, then $e_1,\ldots,e_k$ are pairwise distinct and $v_0=v_k$, implying that $e_k$ and $e_1$ are adjacent in $L$. If $k \ge 3$, we may conclude that $W_L^*$ is a cycle in $L$.
\end{enumerate}
\vspace{-10mm}
\end{proof}

\begin{theo}\label{the:L(H)a}
Let $H=(V,E)$  be a hypergraph with at least 3 edges, and $L=\L(H)$ its intersection graph.
\begin{enumerate}
\item If $H$ is eulerian, then $L$ has a Hamilton cycle.
\item If $H$ is quasi-eulerian, then $L$ has a spanning subgraph whose connected components are 1-regular or 2-regular.
\item If $H$ has an Euler family with no strict closed trail of length less than 3, then $L$ has a 2-factor.
\end{enumerate}
\end{theo}

\begin{proof}
Let $W$ be an Euler tour in $H$. By Lemma~\ref{lem:line}, since $H$ has at least 3 edges, the corresponding sequence $W_L^*$ of vertices in $L$ is a cycle, and since $W$ traverses each edge of $H$ exactly once, the cycle $W_L^*$ traverses each vertex of $L$ exactly once. Thus, $W_L^*$ is a Hamilton cycle of $L$.

Similarly, let ${\cal F}$ be an Euler family of $H$. Each closed strict trail in ${\cal F}$ of length 2 gives rise to a path of length 1 in $L$, while each closed strict trail of length at least 3 corresponds to a cycle in $L$. Since the closed strict trails in ${\cal F}$ are pairwise edge-disjoint, the corresponding subgraphs in $L$ are pairwise vertex-disjoint, and since the members of ${\cal F}$ jointly cover all the edges of $H$, the corresponding subgraphs in $L$ form a spanning subgraph whose connected components are 1-regular or 2-regular. If  ${\cal F}$ contains no strict closed trails of length 2, then this spanning subgraph is in fact a 2-factor.
\end{proof}

Note that in general, the converse of Theorem~\ref{the:L(H)a} does not hold: for a walk $W$ in $\L(H)$, it may happen that every corresponding sequence of vertices and edges in $H$ contains two consecutive vertices that are the same. In Theorem~\ref{the:L(H)b} below, however, we present three families of hypergraphs for which the converse does hold. But first, we need the following lemma.

\begin{lemma}\label{lem:line2}
Let $H=(V,E)$  be a hypergraph and $L=\L(H)$ its intersection graph. Furthermore, let
$W_L= e_0 e_1  \ldots e_{k-1} e_0$ be a cycle in $L$. Assume that one of the following hold:
\begin{description}
    \item[(a)] $\deg_H(v) \le 2$ for all $v \in V$; or
    \item[(b)] $k$ is even and $|e_i \cap e_{i+1}| \ge 2$ for all $i\in \ZZ_k$; or
    \item[(c)] $|e_i \cap e_{i+1}| \ge 2$ for all $i\in \ZZ_k$, and $|e_{k-2} \cap e_{k-1}| \ge 3$.
\end{description}
Then there exist $v_0, v_1, \ldots, v_{k-1} \in V$ such that  $W=v_0 e_0 v_1 e_1 v_2 \ldots v_{k-1} e_{k-1} v_0$ is a closed strict trail in $H$.
\end{lemma}

\begin{proof}
It suffices to choose, for each $i \in \ZZ_k$, a vertex $v_i \in e_{i-1} \cap e_{i}$  such that $v_i \ne v_{i-1}$. Then $W=v_0 e_0 v_1 e_1 v_2 \ldots v_{k-1} e_{k-1} v_0$ will be a closed strict trail in $H$. Consider the following algorithm:
\begin{enumerate}
\item Choose any $v_0 \in e_{k-1} \cap e_0$.
\item For all $i=1,\ldots,k-2$, let $v_i \in e_{i-1} \cap e_{i} - \{ v_{i-1} \}$.
\item Choose $v_{k-1} \in e_{k-2} \cap e_{k-1} - \{ v_0,v_{k-2} \}$.
\end{enumerate}
Steps 1--2 of the algorithm will be successful with any of the assumptions (a), (b), and (c), since $\deg_H(v) \le 2$ for all $v \in V$, or $|e_i \cap e_{i+1}| \ge 2$ for all $i\in \ZZ_k$. Step 3 will also be successful in Cases (a) and (c) since $\deg_H(v) \le 2$ for all $v \in V$, or $|e_{k-2} \cap e_{k-1}| \ge 3$. Hence consider Step 3 in Case (b).

If Step 3 cannot be executed, then we must have that $e_{k-2} \cap e_{k-1}=\{ v_0,v_{k-2} \}$. Consider the subgraph of the cycle $W_L$ induced by the edges of the form $e_{i-1} e_{i}$ such that $e_{i-1} \cap e_{i}=\{ v_0,v_{k-2} \}$. This subgraph is either the (even-length) cycle $W_L$ itself, or it is a vertex-disjoint union of paths. In either case, its edges can be alternately labelled with vertices $v_0$ and $v_{k-2}$, resulting in a revised choice of vertices $v_0,v_1,\ldots,v_{k-1} \in V$ that yields a closed strict trail in $H$.
\end{proof}

\begin{theo}\label{the:L(H)b}
Let $H=(V,E)$  be a hypergraph. Assume $L$ is a graph satisfying one of the following:
\begin{description}
    \item[(a)] $L=\L(H)$ if $\deg_H(v) \le 2$ for all $v \in V$; or
    \item[(b)] $L=\L_2^*(H)$ and $L$ is bipartite; or
    \item[(c)] $L=\L_3^*(H)$.
\end{description}
Then $H$ is eulerian (quasi-eulerian) whenever $L$ has a Hamilton cycle (2-factor, respectively). Moreover, in Cases (b) and (c), $H$ is quasi-eulerian whenever $L$ has a spanning subgraph whose connected components are 1-regular or 2-regular.
\end{theo}

\begin{proof}
Assuming one of the Conditions (a)--(c), by Lemma~\ref{lem:line2}, a cycle in $L$ corresponds to a closed strict  trail in $H$. Moreover, a Hamilton cycle in $L$ corresponds to a closed strict trail containing all the edges of $H$ (that is, an Euler tour), and a 2-factor corresponds to a family of edge-disjoint closed strict  trails that jointly traverse all the edges. Sequentially concatenating any closed strict trails in this family with a common anchor vertex we obtain an Euler family of $H$.

In Cases (b) and (c), each 1-regular component of a spanning subgraph of $L$ gives rise to closed strict trail of length 2 in $H$, and an Euler family is obtained as above.
\end{proof}



\subsection{Characterization in terms of the incidence graph}\label{sec:G}

The following characterization of eulerian and quasi-eulerian hypergraphs in terms of their incidence graph will be henceforth our main tool.

\begin{theo}\label{the:ET-G}
Let $H=(V,E)$ be a connected hypergraph and $G$ its incidence graph. Then
$H$ is quasi-eulerian if and only if $G$ has an even subgraph  that is  2-regular on $E$, and it is eulerian if and only if $G$ has such a  subgraph  with a single non-trivial connected component.
\end{theo}

\begin{proof}
Assume $H$ has an Euler family ${\cal F}$.  Then each edge of $H$ is traversed exactly once by a closed strict trail in ${\cal F}$. Hence ${\cal F}$ corresponds to a family ${\cal F}_G$ of closed trails of $G$ such that each $e \in E$ is traversed exactly once by a trail in ${\cal F}_G$. Let $G'$ be the  subgraph of $G$ corresponding to ${\cal F}_G$.  Then clearly $G'$ is even on $V$ and 2-regular on $E$.

Similarly, if $H$ has an Euler tour $T$, then $T$ is a closed strict trail that traverses each edge of $H$ exactly once, and hence  corresponds to a closed trail $T_G$ of $G$ that traverses each $e \in E$ exactly once. Then $T_G$ may be viewed as the unique connected component of a  subgraph $G'$ of $G$ that is even on $V$ and 2-regular on $E$.

Conversely, suppose $G'$ is an even subgraph of $G$ that is 2-regular on $E$.  Then each non-trivial connected component of $G'$ has an Euler tour; let ${\cal T}$ be the family of these closed trails in $G$.
The closed trails in ${\cal T}$ are pairwise vertex-disjoint and jointly traverse each e-vertex exactly once. Hence ${\cal T}$ corresponds to a family ${\cal F}$ of closed strict trails in $H$ that are pairwise anchor-disjoint and edge-disjoint, and together traverse each edge of $H$ exactly once; that is, an Euler family of $H$. If $G'$ has a single non-trivial connected component, then ${\cal F}$ contains a single closed strict trail; that is, an Euler tour of $H$.
\end{proof}

Using Theorem~\ref{the:ET-G}, certain families of hypergraphs can be easily seen to be eulerian or quasi-eulerian. The first of the following corollaries is immediate.

\begin{cor}\label{cor:2-factor}
Let $H$ be a hypergraph with the incidence graph $G$. If $G$ has a 2-factor, then $H$ is quasi-eulerian. If $G$ is hamiltonian, then $H$ is eulerian.
\end{cor}

\begin{cor}
Let $H$ be an $r$-regular $r$-uniform hypergraph for $r \ge 2$. Then $H$ is quasi-eulerian.
\end{cor}

\begin{proof}
The incidence graph $G$ of $H$ is an $r$-regular bipartite graph with $r \ge 2$. Therefore, as a corollary of Hall's Theorem \cite{Hal}, $G$ admits two edge-disjoint perfect matchings, and hence a 2-factor. Thus $H$ is quasi-eulerian by Corollary~\ref{cor:2-factor}.
\end{proof}

\begin{cor}
Let $H$ be a $2k$-uniform even hypergraph. Then $H$ is quasi-eulerian.

Moreover, $H$ has a collection of Euler families $\{{\cal F}_1,\ldots, {\cal F}_k\}$ such that each flag of $H$ occurs as an anchor flag of exactly one family ${\cal F}_i$ in this collection.
\end{cor}

\begin{proof}
Let $G$ be the incidence graph of $H$. In $G$, every e-vertex has degree $2k$, and every v-vertex has even degree. A result by Hilton \cite[Theorem 8]{Hil} then shows that $G$ has an {\em evenly equitable} $k$-edge colouring; that is, a $k$-edge colouring such that (i) every vertex is incident with an even number of edges of each colour, and (ii) for each vertex, the numbers of edges of any two colours that are incident with this vertex differ by at most two. Hence the $i$-th colour class, for $i=1,2,\ldots,k$, induces an even subgraph $G_i$ of $G$ that is 2-regular on $E$. By Theorem~\ref{the:ET-G}, each $G_i$ corresponds to an Euler family ${\cal F}_i$ of $H$, and $H$ is quasi-eulerian. Since every edge of $G$ lies in exactly one of $G_1,\ldots,G_k$, it follows that each flag of $H$ occurs as an anchor flag of exactly one family among ${\cal F}_1,\ldots, {\cal F}_k$.
\end{proof}


\subsection{Characterization using blocks}\label{sec:blocks}

In this section, we reduce the problem of existence of an Euler family in a hypergraph to the identical problem on its blocks (to be defined below, analogously to blocks in graphs). As expected, this reduction is a little more complicated, and perhaps not as useful, in the case of Euler tours. We refer the reader to \cite{BahSaj} for more information on blocks in a hypergraph.

\begin{defn} {\rm \cite{BahSaj}
Let $H=(V,E)$ be a connected hypergraph without empty edges. A vertex $v \in V$ is a {\em separating vertex} for $H$ if $H$ decomposes into two non-empty connected hypersubgraphs with just vertex $v$ in common. That is, $H=H_1 \oplus H_2$, where  $H_1$ and $H_2$ are two non-empty connected hypersubgraphs of $H$ with  $V(H_1) \cap V(H_2)=\{ v \}$.
}
\end{defn}

\begin{defn} {\rm \cite{BahSaj}
A connected hypergraph without empty edges that has no separating vertices is called {\em non-separable}. A {\em block} of a hypergraph $H$ is a maximal non-separable hypersubgraph of $H$.
}
\end{defn}

\begin{theo}\label{the:blocks}
Let $H=(V,E)$ be a hypergraph. Then:
\begin{enumerate}
\item $H$ has an Euler family if and only if each block of $H$ has an Euler family.
\item $H$ has an Euler tour (necessarily  traversing every separating vertex of $H$) if and only if each block $B$ of $H$ has an Euler tour that traverses every separating vertex of $H$ contained in $B$.
\end{enumerate}
\end{theo}

\begin{proof}
Let $B$ be any block of $H$, and $T$ a closed strict trail of $H$. In $T$, delete all vertices and edges of $H$ that are not in $B$. By \cite[Theorem 3.36]{BahSaj}, every cycle of $H$ is contained within a block; consequently, each of the remaining subsequences of $T$ ends with the first vertex of the next subsequence (in cyclical order). Denote the concatenation of these remaining subsequences of $T$ by $T\vert_B$, and observe that $T\vert_B$ is a closed strict trail in $B$.
\begin{enumerate}
\item Assume $H$ has an Euler family $\cal F$, and let $B$ be any block of $H$. Then the set of all closed strict trails of the form $T\vert_B$, for all trails $T$ in  ${\cal F}$ that contain edges of $B$, is an Euler family of $B$.

    Conversely, suppose that each block $B$ of $H$ has an Euler family ${\cal F}_B$. Take the union of all ${\cal F}_B$ and pairwise concatenate any closed strict trails in this union that have an anchor in common, until every pair of resulting trails are anchor-disjoint. The result is an Euler family of $H$.
\item Assume $H$ has an Euler tour $T$ (necessarily  traversing every separating vertex of $H$), and let $B$ be a block of $H$. Then $T\vert_B$ is an Euler tour of the block $B$ traversing  every separating vertex of $H$ contained in $B$.

    Conversely, if each block $B$ of $H$ has an Euler tour $T_B$ traversing  every separating vertex of $H$ contained in $B$, then all the $T_B$ can be concatenated to give an Euler tour of $H$.
\end{enumerate}
\vspace{-10mm}
\end{proof}

Using our characterization in terms of the incidence graph (Theorem~\ref{the:ET-G}), the above theorem can be immediately augmented as follows.

\begin{cor}
Let $H=(V,E)$ be a hypergraph. Then:
\begin{enumerate}
\item $H$ has an Euler family if and only if for each block $B$ of $H$, the incidence graph $G_B$ of $B$ has an even subgraph $G_B'$ that is 2-regular on $E(B)$.
\item $H$ has an Euler tour  if and only if for each block $B$ of $H$, the incidence graph $G_B$  of $B$ has an even subgraph $G_B'$ that is 2-regular on $E(B)$ and has a unique non-trivial connected component, which necessarily contains every separating vertex of $H$ that lies in $B$.
\end{enumerate}
\end{cor}

With the insight of Theorem~\ref{the:blocks}, one can easily construct a connected hypergraph $H$ that is quasi-eulerian but not eulerian, namely, one with a separating vertex. Let $B$ be an eulerian non-separable hypergraph that has more vertices than edges, and let $A$ be any eulerian hypergraph. For each vertex $v$ of $B$, let $A_v$ be a copy of the hypergraph $A$ such that $V(A_v) \cap V(B)=\{ v \}$ and the hypergraphs $A_v$ are all pairwise vertex-disjoint. Now let $H$ be the union of $B$ with all $A_v$. Each $v \in V(B)$ is now a separating vertex of $H$, and since $B$ cannot have an Euler tour traversing every vertex, $H$ has no Euler tour. However, the  Euler tours of $B$ and all the $A_v$ (concatenating any that have common anchor vertices) will give rise to an Euler family for $H$.

One may then ask whether a connected hypergraph without separating vertices that admits an Euler family necessarily admits an Euler tour. The answer is negative, as shown by the counterexample below.

\begin{example}{\rm
Let $H$ be a hypergraph whose incidence graph is shown in Figure~\ref{fig:Count1}. Observe that $H$ is quasi-eulerian but not eulerian, has no cut edges and no separating vertices, and  satisfies the necessary condition  from Lemma~\ref{lem:NecCond}.
}
\end{example}

\begin{center}
\begin{figure}[h!]
\centerline{\includegraphics[scale=0.5]{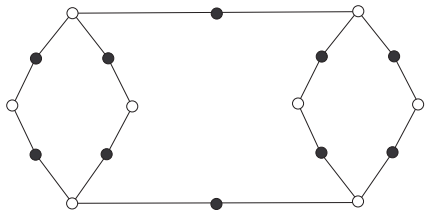}}
\caption{The incidence graph of a quasi-eulerian but not eulerian hypergraph, without cut edges and without separating vertices.  (Black dots represent the v-vertices.)}\label{fig:Count1}
\end{figure}
\end{center}


\subsection{Complexity of the {\sc Euler Tour} and {\sc Euler Family} problems}\label{sec:compl}

We shall now turn our attention to the complexity of the two problems. First, we define the {\em bigness} of a hypergraph  as the maximum of the order (number of vertices) and the size (number of edges) of the hypergraph. In this section, we show that the problem of determining whether or not a given hypergraph is eulerian is NP-complete, while --- perhaps surprisingly --- the problem of determining whether or not a given hypergraph is quasi-eulerian is polynomial in the bigness of the hypergraph.

We begin by formally defining our first decision problem.

\begin{problem}
{\sc Euler Tour}

{\sc Given:} A hypergraph $H$.

{\sc Decide:} Does $H$ have an Euler tour?
\end{problem}

Recall that a hypergraph is called {\em linear} if every pair of distinct edges intersect in at most one vertex.

Lonc and Naroski \cite{LonNar} showed that {\sc Euler Tour} is NP-complete on the set of $k$-uniform hypergraphs for any $k\ge 3$ (as well as on the set of 3-uniform hypergraphs with a connected skeleton). Their proof for $k=3$ actually shows that the problem is NP-complete on the smaller class of linear 2-regular 3-uniform hypergraphs, as stated in Theorem~\ref{the:ETproblem1} below. For completeness, we include the proof from \cite{LonNar}.
In the polynomial reduction, the following known NP-complete problem is used.

\begin{problem}
{\sc Hamilton Cycle}

{\sc Given:} A graph $G$.

{\sc Decide:} Does $G$ have a Hamilton cycle?
\end{problem}

\begin{theo}\label{the:ETproblem1}
Let ${\cal LH}_3^2$ denote the family of linear 2-regular 3-uniform hypergraphs. Then {\sc  Euler Tour} is NP-complete on ${\cal LH}_3^2$.
\end{theo}

\begin{proof} \cite{LonNar}
Clearly, {\sc  Euler Tour} is in the class NP since a potential solution can be verified in time that is polynomial in the number of edges, and hence also in the bigness of the hypergraph.

Let $G=(V_G,E_G)$ be a simple cubic graph. Define a hypergraph $H=(V_H,E_H)$ as follows: $V_H=E_G$ and $E_H=\{ e_v: v\in V_G \}$, where for each vertex $v \in V_G$, we let $e_v$ be the set of edges of $G$ incident with $v$. (In other words, $H$ is the dual of $G$.) Observe that $H$ is linear, 2-regular, and 3-uniform. Now, if $C=v_0 e_1 v_1 e_2 \ldots v_{n-1} e_n v_0$ is a Hamilton cycle in $G$, then $e_1 e_{v_1} e_2 \ldots e_{v_{n-1}} e_n e_{v_0} e_1$ is an Euler tour of $H$. Conversely, suppose $T=e_0 e_{v_1} e_1 \ldots e_{v_{n-1}} e_{n-1} e_{v_n} e_0$ is an Euler tour of $H$. Then $C=v_1 e_1 \ldots v_{n-1} e_{n-1} v_n e_0 v_1$ is a Hamilton cycle in $G$.

These conversions are clearly polynomial in the number of vertices of $G$, and hence also in the size and bigness of $H$. Therefore, since {\sc Hamilton Cycle} is NP-complete on the set of all cubic graphs \cite{GarJohTar}, {\sc  Euler Tour} is NP-complete on the set ${\cal LH}_3^2$.
\end{proof}

Next, we define our second decision problem.

\begin{problem}
{\sc Euler Family}

{\sc Given:} A hypergraph $H$.

{\sc Decide:} Does $H$ have an Euler family?
\end{problem}

Below, we show that {\sc Euler Family} is a polynomial problem on the set of all hypergraphs. In the reduction, the following known polynomial problem \cite{Edm} will be used.

\begin{problem}
{\sc 1-Factor}

{\sc Given:} A graph $G$.

{\sc Decide:} Does $G$ have a 1-factor?
\end{problem}

For a graph $X$ and a function $f: V(X) \rightarrow \NN$, an {\em $f$-factor} of $X$ is a spanning subgraph $X'$ of $X$ such that $\deg_{X'}(v)=f(v)$ for all $v \in V(G)$. Bondy and Murty \cite{BonMur} describe a polynomial reduction, originally due to Tutte \cite{Tut}, of the problem of existence of an $f$-factor in a graph without loops to the problem of existence of a 1-factor in a graph. This reduction can be extended to graphs with loops as follows.

\begin{lemma}\label{lem:EFconversion}
Let $X=(V,E)$ be a graph obtained from a simple graph by adjoining $\ell(u)$ loops to each vertex $u$, where $\ell(u)$ is polynomial in the order of $X$. Let $f: V \rightarrow \NN$ be a function with $f(u) \le \deg_X(u)$ for all $u \in V$.

For each $v \in V$, construct a graph $Y_v$ as follows:
\begin{itemize}
\item The vertex set of $Y_v$ has a partition $\{ S_v, T_v, U_v \}$, with $|S_v|=\deg_{X}(v)-f(v)$, $|T_v|=\deg_{X}(v)-2\ell(v)$, and $|U_v|=2\ell(v)$.
\item The edge set of $Y_v$ consists of all edges of the form $uv$ for $u\in S_v$ and $v \in T_v \cup U_v$, as well as a perfect matching on the set $U_v$.
\end{itemize}
A graph $X_f$ is obtained from $X$ by taking the vertex-disjoint union of the graphs $Y_v$, for all $v \in V$, and inserting a single linking edge with one endpoint in $T_u$ and the other in $T_v$ if and only if $uv \in E$ and $u \ne v$. The endpoints of these edges are chosen so that each vertex in $T_v$, for each $v \in V$, is an endpoint of exactly one of these linking edges.

A 1-factor in $X_f$ then corresponds to an $f$-factor in $X$, and vice-versa, and this conversion is polynomial in the order of $X$.
\end{lemma}

\begin{proof}
First, observe that the graph $X_f$ is well defined: each vertex $v \in V$ has exactly  $\deg_X(v)-2\ell(v)=|T_v|$ neighbours $u \ne v$ in $X$, and so it is indeed possible to join each vertex of $T_v$ to exactly one vertex in some $T_u$ such that $u \ne v$ and $uv \in E$.

Next, we show that a 1-factor in $X_f$ corresponds to an $f$-factor in $X$. Let $M_f$ be the edge set of a 1-factor (that is, a perfect matching) in $X_f$. For each $v \in V$, let $\ell_f(v)$ denote the number of edges of $M_f$ with both ends in $U_v$. Then let $E'$ be the subset of $E$ containing all edges
$uv$ such that $u, v \in V$, $u \ne v$, for which there exist $x \in T_u$ and $y \in T_v$ with $xy\in M_f$; in addition, let $E'$ contain  exactly $\ell_f(v)$ loops incident with $v$ for each vertex $v \in V$.

Let $F=(V,E')$. We claim $F$ is an $f$-factor of $X$. Fix any vertex $v\in V$. Let $\tau(v)$ be the number of edges of $M_f$ with one endpoint in $T_v$ and the other in a set $T_u$ for all $u \ne v$, and observe that $\deg_F(v)=\tau(v) +2\ell_f(v)$. Furthermore, let $\sigma(v)$ be the number of edges of $M_f$ with one endpoint in $T_v$ and the other in $S_v$, and let $\nu(v)$ be the number of edges of $M_f$ with one endpoint in $S_v$ and the other in $U_v$. Then
\begin{eqnarray*}
|T_v| &=& \deg_X(v)-2\ell(v)=\tau(v)+\sigma(v), \\
|S_v| &=& \deg_X(v)-f(v) = \sigma(v)+\nu(v), \\
|U_v|&=& 2\ell(v)= \nu(v)+ 2\ell_f(v),
\end{eqnarray*}
which yields $f(v)=\tau(v)+2\ell_f(v)=\deg_F(v)$. We conclude that, indeed, $F$ is an $f$-factor of $X$.

Conversely, take any $f$-factor $F$ of $X$, and construct a subset $M_f$ of edges of $X_f$ as follows.
\begin{itemize}
\item[(1)] For each $u,v \in V$, $u \ne v$, if $uv \in E(F)$, then let $M_f$ contain the unique edge with one endpoint in $T_u$ and the other in $T_v$.
\item[(2)] For each $v \in V$, if $E(F)$ contains $\ell_f(v)$ loops incident with $v$, then let $M_f$ contain $\ell_f(v)$ edges with both ends in $U_v$.
\item[(3)] For each $v \in V$, let $M_f$ contain $|S_v|$ independent edges with one endpoint in $S_v$ and the other in $U_v \cup T_v$.
\end{itemize}
We now show that the edges  in (3) can be chosen to be independent from the edges chosen in (1) and (2). Indeed, in (1), for each $v \in V$, $f(v)-2\ell_f(v)$ edges with one endpoint in $T_v$ were chosen, and in (2), $2\ell_f(v)$ edges with both ends in $U_v$ were put into $M_f$. This leaves
$$|T_v|-\left(f(v)-2\ell_f(v)\right)=\left(\deg_X(v)-2\ell(v)\right)-\left(f(v)-2\ell_f(v)\right)$$ vertices in $T_v$, and
$$|U_v|-2\ell_f(v)=2\ell(v)-2\ell_f(v)$$
vertices in $U_v$ unsaturated.
Since
$$\left(\deg_X(v)-2\ell(v)\right)-\left(f(v)-2\ell_f(v)\right)+\left(2\ell(v)-2\ell_f(v)\right)=\deg_X(v)-f(v)=|S_v|,$$
the edges in (3) can be chosen so that $M_f$ is an independent set and  every vertex in $X_f$ is $M_f$-saturated. We conclude that $(V(X_f),M_f)$ is a 1-factor in $X_f$.

Since for each $v \in V$, the number of loops $\ell(v)$ incident with $v$ is polynomial in the order of $X$, these conversions are polynomial in the order of $X$.
\end{proof}


\begin{theo}\label{the:EFcomplexity}
Let ${\cal H}$ be the family of all hypergraphs. Then {\sc Euler Family} is polynomial on ${\cal H}$.
\end{theo}

\begin{proof}
Let $H=(V,E)$ be a hypergraph, and $G$ its incidence graph. By Theorem~\ref{the:ET-G}, $H$ admits an Euler family if and only if $G$ has an even  subgraph $G'$ that is 2-regular on $E$; in this proof, we shall call such a subgraph $G'$ an {\em EF-factor} of $G$.

Starting from $G$, construct a graph $G^*$ by appending $\lfloor \frac{\deg_G(v)}{2} \rfloor$ loops to each $v \in V$. Then define a function $f: V \cup E \rightarrow \NN$ by $f(e)=2$ for all $e \in E$, and $f(v)=2\lfloor \frac{\deg_G(v)}{2} \rfloor$ for all $v \in V$.

We claim that $G$ has an EF-factor $G'$ if and only if $G^*$ has an $f$-factor. Indeed, take an EF-factor $G'$ of $G$. Appending $\frac{1}{2}(f(v)-\deg_{G'}(v))$ loops to each vertex $v \in V$ results in an $f$-factor of $G^*$. Conversely, removing the loops from any $f$-factor of $G^*$ will result in an EF-factor $G'$ of $G$. This conversion is clearly polynomial in the order of $G$, and hence in the bigness of $H$.

By Lemma~\ref{lem:EFconversion} and \cite{Edm}, the problem of finding an $f$-factor in the graph $G^*$ is polynomial in the order of $G^*$. Hence the problem of finding an Euler family in $H$ is polynomial in the bigness of $H$. We conclude that {\sc Euler Family} is polynomial on ${\cal H}$.
\end{proof}

Note that the reduction to the problem of a $1$-factor in a graph as described in Lemma~\ref{lem:EFconversion}, together with Edmonds' Algorithm \cite{Edm} for finding a maximum matching in an arbitrary graph, gives us a polynomial-time algorithm for constructing an Euler family in a quasi-eulerian hypergraph. We should also mention that in the proof of Theorem~\ref{the:EFcomplexity}, instead of Lemma~\ref{lem:EFconversion}, we could have used the fact that the general  $f$-factor problem  is polynomial \cite[Theorem 6.2]{AnsJAlg85}.


\subsection{Quasi-eulerian hypergraphs: necessary and sufficient conditions}\label{NAC}

Theorem~\ref{the:ET-G} shows that a hypergraph $H=(V,E)$ admits an Euler family if and only if its incidence graph $\G(H)$ has an even  subgraph  that is 2-regular on $E$. We shall now combine this observation with Lovasz's Theorem~\ref{the:Lovasz} (below) to give more easily verifiable necessary and sufficient conditions.

For a graph $G$ and functions $f,g: V(G) \rightarrow \NN$,  a {\em $(g,f)$-factor} of $G$ is a spanning subgraph $F$ of $G$ such that $g(x) \le \deg_F(x) \le f(x)$ for all $x \in V(G)$. An $f$-factor is then simply an $(f,f)$-factor.  For any subgraph $G_1$ of $G$ and any sets $V_1, V_2 \subseteq V(G)$, let $\eps_{G_1}(V_1,V_2)$ denote the number of edges of $G_1$ with one end in $V_1$ and the other in $V_2$.

\begin{theo} \cite{Lov} \label{the:Lovasz}
Let G be a graph and let $f,g : V (G) \rightarrow \NN$ be functions such that
$g(x) \le f(x)$ and $g(x) \equiv f(x) \pmod{2}$ for all $x \in V (G)$. Then G has a $(g, f)$-factor $F$ such that $\deg_F(x) \equiv f(x) \pmod{2}$ for all $x \in V(G)$ if and only if all disjoint subsets $S$ and $T$ of $V(G)$ satisfy
\begin{equation}\label{eq:Lov}
\sum_{x \in S} f(x) + \sum_{x \in T} (\deg_G(x)-g(x))- \eps_G(S,T)-q(S,T) \ge 0,
\end{equation}
where $q(S,T)$ is the number of connected components $C$ of $G \setminus (S \cup T)$ such that $$\sum_{x \in V(C)} f(x) + \eps_G(V(C),T) \quad \mbox{is odd}.$$
\end{theo}

\begin{cor}\label{cor:Lovasz}
Let $H=(V,E)$ be a hypergraph and $G=\G(H)$ its incidence graph. Then $H$ is quasi-eulerian if and only if  all disjoint sets $S \subseteq E$ and $T\subseteq V \cup E$ of $V(G)$ satisfy
\begin{equation}\label{eq:LovH}
2|S| + \sum_{x \in T} \deg_G(x)  - 2|T \cap E|  - \eps_G(S,T \cap V)-q(S,T) \ge 0,
\end{equation}
where $q(S,T)$ is the number of connected components $C$ of $G \setminus (S \cup T)$ such that $\eps_G(V(C),T)$ is odd.
\end{cor}

\begin{proof}
By Theorem~\ref{the:ET-G}, $H$ has an Euler family if and only if $G$ has an even subgraph $G'$ that is 2-regular on $E$. Define functions $f,g: V \cup E \rightarrow \NN$ as follows:
$$g(x)=\left\{ \begin{array}{ll}
              0 & \mbox{ if } x \in V \\
              2 & \mbox{ if } x \in E
              \end{array} \right.
\qquad \mbox{ and } \qquad
f(x)=\left\{ \begin{array}{ll}
              K & \mbox{ if } x \in V \\
              2 & \mbox{ if } x \in E
              \end{array} \right.
              ,$$
where $K$ is a sufficiently large even integer.
Observe that $f$ and $g$ satisfy the assumptions of Theorem~\ref{the:Lovasz}. Moreover, a  subgraph $G'$ of $G$ with the required properties is a $(g,f)$-factor $F$ of $G$ with $\deg_F(x) \equiv f(x) \pmod{2}$ for all $x \in V(G)$, and conversely.

For any subsets $S$ and $T$ of $V \cup E$, if $S \cap V \ne \emptyset$, then $\sum_{x \in S} f(x)$ is very large, and Condition~(\ref{eq:Lov}) clearly holds for $S$ and $T$. Thus Theorem~\ref{the:Lovasz} asserts that $G$ has an $(f,g)$-factor if and only if Condition~(\ref{eq:Lov}) holds for all disjoint sets $S \subseteq E$ and $T\subseteq V \cup E$ of $V(G)$.

Observing that $\sum_{x \in V(C)} f(x) + \eps_G(V(C),T) \equiv \eps_G(V(C),T) \pmod{2}$, it is then straightforward to show that Condition~(\ref{eq:Lov}) in Theorem~\ref{the:Lovasz} is equivalent to Condition~(\ref{eq:LovH}) in the statement of this corollary. The result follows as claimed.
\end{proof}

To express the necessary conditions in Corollary~\ref{cor:Lovasz} in the language of the hypergraph itself, we introduce the following term. For a hypergraph $H=(V,E)$ and sets $V' \subseteq V$ and $E' \subseteq E$, the symbol $H[V',E']$ will denote the hypergraph with vertex set $V'$ and edge set $\{ e \cap V': e \in E'\}$. Observe that the incidence graph of $H[V',E']$ is then the subgraph of $\G(H)$ induced by the vertex set $V' \cup E'$.

\begin{cor}\label{cor:LovaszH}
A hypergraph $H=(V,E)$  is quasi-eulerian if and only if every subset $V' \subseteq V$ and all disjoint subsets $E', E''\subseteq E$ satisfy
\begin{equation}\label{eq:LovH2}
2|E''| + \sum_{v \in V'} \deg_H(v) + \sum_{e \in E'} |e| - 2|E'| - |F(H[V',E''])| - q_H(E'',V' \cup E') - q_e(V') \ge 0,
\end{equation}
where $q_H(E'',V' \cup E')$ is the number of connected components $C$ of $(H-(E' \cup E'')) \setminus V'$ such that
$|F(H[V(C),E'])|+|F(H[V',E(C)])|$ is odd, and
$q_e(V')$ is the number of edges $e \in E-(E' \cup E'')$ such that $e \subseteq V'$ and $|e|$ is odd.
\end{cor}

\begin{proof}
Let $G$ be the incidence graph of $H$.
It suffices to show that the condition in Corollary~\ref{cor:LovaszH} holds if and only if the condition in Corollary~\ref{cor:Lovasz} holds.
Take any subset $V' \subseteq V$ and disjoint subsets $E', E''\subseteq E$, and let $S=E''$ and $T=V' \cup E'$.  Clearly,
$$2|S| + \sum_{x \in T} \deg_G(x)  - 2|T \cap E|=
2|E''|+\sum_{v \in V'} \deg_H(x) + \sum_{e \in E'} |e| - 2|E'|.$$
Next, we have
$$ \eps_G(S,T \cap V)=\eps_G(E'',V')=|F(H[V',E''])|.$$
Observe that the incidence graph of $(H-(E' \cup E''))\setminus V'$ is obtained from
$G \setminus (S \cup T)=G\setminus (E'' \cup V' \cup E')$ by deleting any isolated e-vertices; these are precisely the edges $e \in E-(E'\cup E'')$ such that $e \subseteq V'$.
Hence, by Theorem~\ref{the:conn}, the connected components of $G \setminus (S \cup T)$ are either the incidence graphs of the connected components of $(H-(E' \cup E''))\setminus V'$, or else correspond to the edges $e \in E-(E'\cup E'')$ such that $e \subseteq V'$.
Take any connected component $C_G$ of $G \setminus (S \cup T)$. If $C_G$ is the incidence graph of a connected component $C$ of $(H-(E' \cup E''))\setminus V'$, then
$$\eps_G(V(C_G),T)=\eps_G(V(C)\cup E(C),V' \cup E')=|F(H[V(C),E'])|+
|F(H[V',E(C)])|.$$
If however, $C_G$ corresponds an isolated e-vertex $e$, then
$$\eps_G(V(C_G),T)=\eps_G(\{e\},V' \cup E')=|e|.$$
Thus
$$q(S,T)=q(E'',V' \cup E')=q_H(E'',V' \cup E') + q_e(V'),$$
and Conditions~(\ref{eq:LovH}) and (\ref{eq:LovH2}) are equivalent.
\end{proof}

Using Theorem~\ref{the:blocks}, we immediately obtain the following.

\begin{cor}
A hypergraph $H$ is quasi-eulerian if and only if the necessary and sufficient condition in Corollary~\ref{cor:LovaszH} holds for every block of $H$.
\end{cor}


\subsection{Quasi-eulerian 3-uniform hypergraphs}\label{sec:EF3}

In Theorem~\ref{the:cut-edges}, we saw that a hypergraph with strong cut edges cannot be quasi-eulerian, while the examples in Figures~\ref{fig:pic4} and \ref{fig:pic5} show that a hypergraph with cut edges, none of which is strong, may or may not be quasi-eulerian. The following theorem completes the picture for 3-uniform hypergraphs. The main ingredient in the proof  is the following result by Fleischner.

\begin{theo}\cite{Fle}\label{the:Fle}
Every graph without cut edges and of minimum degree at least 3 has a spanning even subgraph without isolated vertices.
\end{theo}

\begin{theo}\label{the:3-uniform}
Let $H=(V,E)$ be a 3-uniform hypergraph without cut edges. Then $H$ is quasi-eulerian.
\end{theo}

\begin{proof}
Let $G=\G(H)$ be the incidence graph of $H$. We claim that, since $H$ has no cut edges, the graph $G$ has no cut edges. Suppose, to the contrary, that $ve$ is a cut edge of $G$ (where $v\in V$ and $e \in E$), and let $G_v$ and $G_e$ be the connected components of $G-ve$ containing vertex $v$ and $e$, respectively. Since $|e|>1$, the component $G_e$ must contain a v-vertex $w$, and $v$ and $w$ are disconnected in $G-ve$. Hence they are disconnected in $H-e$, showing that $e$ is a cut edge of $H$, a contradiction. Therefore $G$ has no cut edges as claimed.

Note that we may assume that $H$, and hence $G$,  has no isolated vertices. Clearly, $G$ has no vertices of degree 1, since the edge incident with such a vertex would necessarily be a cut edge. Suppose $G$ has a vertex of degree 2. Then it must be a v-vertex, since $|e|=3$ for all $e \in E$. Obtain a graph $G^*$ from $G$ by replacing, for every vertex $v$ of degree 2, the 2-path $e_1ve_2$ in $G$ with an edge $e_1e_2$. Observe that in $G^*$, all e-vertices have degree 3, and all v-vertices have degree at least 3. Moreover, $G^*$ has no cut edges since $G$ does not. Therefore, by Theorem~\ref{the:Fle}, $G^*$ has a spanning even subgraph $G_1^*$ without isolated vertices. We construct a subgraph $G_1$ of $G$ as follows: for any vertex $v$ of degree 2 in $G$, and its incident edges $e_1$ and $e_2$, if $e_1e_2$ is an edge of $G_1^*$, then replace it with the 2-path $e_1ve_2$. The resulting graph $G_1$ is an even subgraph of $G$ without isolated e-vertices. Since every e-vertex of $G$ has degree 3 in $G$, it has degree 2 in $G_1$. Thus, by Theorem~\ref{the:ET-G}, $G_1$ gives rise to an Euler family of $H$.
\end{proof}

The reader may have noticed that an Euler family of $H$ constructed in the proof of Theorem~\ref{the:3-uniform} traverses every vertex of $H$ except possible some of the vertices of degree 2. Observe that Theorem~\ref{the:3-uniform} does not hold for graphs (that is, 2-uniform hypergraphs); an example is a cycle with a chord. More generally, it does not hold for all hypergraphs in which every edge has size 2 or 3; such an example is given in Figure~\ref{fig:pic6}.

\begin{center}
\begin{figure}[h!]
\centerline{\includegraphics[scale=0.5]{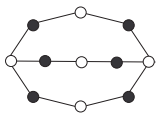}}
\caption{The incidence graph of a hypergraph without cut edges that is not quasi-eulerian; observe that every edge has size 2 or 3.  (Black dots represent the v-vertices.)} \label{fig:pic6}
\end{figure}
\end{center}

\begin{cor}
Let $H=(V,E)$ be a 3-uniform hypergraph with at least two edges such that each pair of vertices lie together in at least one edge. Then $H$ is quasi-eulerian.
\end{cor}

\begin{proof}
By Theorem~\ref{the:3-uniform}, it suffices to show that $H$ has no cut edges. If $|V|=3$, then clearly none of the edges are cut edges. Hence assume $|V|\ge 4$, and suppose that $e \in E$ is a cut edge of $H$. Let $u_1, u_2 \in V$ be vertices of $e$ that lie in distinct connected components of $H-e$, and consider any vertex $w \not\in e$. Then there exist edges $e_1,e_2$ such that $w,u_i \in e_i$ for $i=1,2$. Since obviously $e_1,e_2 \ne e$, vertex $w$ must lie in the same connected component of $H-e$ as both $u_1$ and $u_2$, contradicting the fact that $u_1$ and $u_2$ lie in distinct connected components of $H-e$. Hence $H$ has no cut edges, and by Theorem~\ref{the:3-uniform} it is quasi-eulerian.
\end{proof}

Recall that a triple system TS($n$,$\lambda$) is a 3-uniform hypegraph of order $n$ such that every pair of vertices lie together in exactly $\lambda$ edges.

\begin{cor}
Every triple system TS($n$,$\lambda$) with $(n,\lambda) \ne (3,1)$ is quasi-eulerian.
\end{cor}

We mention that Wagner and the second author recently proved that all triple systems TS($n$,$\lambda$), except for TS(3,1), are in fact eulerian \cite{SajWag}.

The proof of \cite[Theorem 2]{LonNar} for the case $k=3$ is very long and technical; as another corollary of our Theorem~\ref{the:3-uniform}, we show that every 3-uniform hypergraph with a connected 2-intersection graph is eulerian provided that it has no pendant vertices.

\begin{cor} Every 3-uniform hypergraph with a connected 2-intersection graph and without pendant vertices is eulerian.
\end{cor}

\begin{proof}
Let $H=(V,E)$ be a 3-uniform hypergraph with a connected 2-intersection graph $L$ and without pendant vertices. Hence $H$ has at least 2 edges. Suppose it has a cut edge $e$. Since $L$ is connected, $e$ shares exactly two of its vertices with another edge; consequently, these two vertices lie in the same connected component of $H-e$. Thus $H-e$ has exactly two connected components;  let $H_1$ be the connected component  containing a single vertex, $w$, of $e$. By assumption, $w$ is not a pendant vertex in $H$, so $E(H_1) \ne \emptyset$. Take any $e_1 \in E(H_1)$ and any $e_2 \in E-E(H_1)$. Then $e_1 \cap e_2 \subseteq \{ w \}$, whence $e_1e_2 \not\in E(L)$. It follows that $L$ is disconnected, a contradiction.

We conclude that $H$ has no cut edges, and hence is  quasi-eulerian by Theorem~\ref{the:3-uniform}.
Let ${\cal F}=\{T_1,\ldots,T_k\}$ be an Euler family of $H$ with a minimum number of components, and suppose $k \ge 2$. Let $G$ be the incidence graph of $H$, let $G'$ be the subgraph of $G$ corresponding to ${\cal F}$, and $G_1,\ldots,G_k$ the connected components of $G'$ corresponding to the closed strict trails $T_1,\ldots,T_k$ of $H$. Since $L$ is connected, without loss of generality, there exist e-vertices $e_1$ of $G_1$ and $e_2$ of $G_2$ that are adjacent in $L$, and hence in $G$ have two common neighbours, say $v_1$ and $v_2$. Since $e_1$ and $e_2$ are of degree 3 in $G$, and of degree 2 in $G'$, each is adjacent to at least one of $v_1$ and $v_2$ in $G'$, and since they lie in distinct connected components of $G'$, we may assume without loss of generality that $v_1e_1, v_2e_2 \in E(G')$. Obtain $G''$ by replacing these two edges of $G'$ with edges $v_1e_2$ and $v_2e_2$. Then $G''$  an even subgraph of $G$ that is 2-regular on $E$, so it corresponds to an Euler family of $H$. But since $G''$ one fewer connected component than $G'$, it contradicts the minimality of ${\cal F}$.

We conclude that $k=1$, that is, $H$ admits an Euler tour.
\end{proof}

We conclude this section with an alternative, more detailed characterization of quasi-eulerian 3-uniform hypergraphs that are either even or odd, in terms of their incidence graph.

\begin{theo}
Let $H=(V,E)$ be an even 3-uniform hypergraph and $G$ its incidence graph. The following are equivalent:
\begin{enumerate}
\item $H$ is quasi-eulerian.
\item $G$ has a  subgraph $G''$ that is even on $V$ and 1-regular on $E$.
\item $E$ can be partitioned into pairs $\{ e, e' \}$ such that $e \cap e' \ne \emptyset$.
\end{enumerate}
\end{theo}

\begin{proof}
(1) $\Leftrightarrow$ (2): Assume $H$ is quasi-eularian. By Theorem~\ref{the:ET-G}, $G$ has an even subgraph $G'$ that is 2-regular on $E$. Define $G''=(V(G),E(G)-E(G'))$. Since $G$ and $G'$ are both even on $V$, so is $G''$, and since $G$ is 3-regular  and $G'$ is 2-regular on $E$, $G''$ is 1-regular on $E$. The converse is proved very similarly.

(2) $\Rightarrow$ (3): Assume $G$ has a  subgraph $G''$ that is even on $V$ and 1-regular on $E$. For each $v \in V$, the set of all $e \in E$ such that $ve \in E(G'')$ is of even cardinality, and hence can be partitioned into pairs $\{ e, e' \}$ such that $v \in e \cap e'$.

(3) $\Rightarrow$ (2): Let  $\cal P$ be such a partition of $E$.  For each pair $\{ e, e' \} \in {\cal P}$,  choose  $v \in e \cap e'$, and let $G''$ be induced by the set of all edges of the form $ve$ and $ve'$. Then $G''$ is even on $V$ and 1-regular on $E$ as required.
\end{proof}

The analogous result for odd hypergraphs (below), is proved similarly.

\begin{theo}\label{the:odd3hgs}
Let $H=(V,E)$ be an odd 3-uniform hypergraph and $G$ its incidence graph. The following are equivalent:
\begin{enumerate}
\item $H$ is quasi-eulerian.
\item $G$ has an odd subgraph $G''$ that is  1-regular on $E$.
\item $E$ can be partitioned into sets $S$ of odd cardinality such that $ \bigcap_{e \in S} e \ne \emptyset$.
\end{enumerate}
\end{theo}

\begin{cor}
Let $H=(V,E)$ be an odd 3-uniform quasi-eulerian hypergraph. Then $|V| \le |E|$.
\end{cor}

\begin{proof}
By Theorem~\ref{the:odd3hgs}, $\G(H)$ has an odd subgraph $G''$ that is  1-regular on $E$. Since every $v \in V$ has degree at least 1 in $G''$, and no two v-vertices can have a common neighbour in $G''$, we must have $|V| \le |E|$.
\end{proof}


\subsection{Cycle decomposition and 2-factors of quasi-eulerian \\ hypergraphs}\label{sec:CD}

The well-known Veblen's Theorem \cite{Veb} states that a connected graph is even (and hence eulerian) if and only if it admits a decomposition into cycles. The analogous result for hypergraphs is presented below.

\begin{theo}\label{the:CD}
A hypergraph is quasi-eulerian if and only if it admits a decomposition into cycles.
\end{theo}

\begin{proof}
Let $H=(V,E)$ be a quasi-eulerian hypergraph and $G$ its incidence graph. By Theorem~\ref{the:ET-G},  $G$ has an even subgraph $G'$ that is 2-regular on $E$. Hence $G'$ admits a decomposition into cycles, ${\cal C}_G$, and every e-vertex lies in exactly one of the cycles in ${\cal C}_G$. Let ${\cal C}_H$ be the corresponding family of cycles in $H$. Then every $e \in E$ lies in exactly one of the cycles in ${\cal C}_H$, so ${\cal C}_H$ is a cycle decomposition of $H$.

Conversely, assume that $H=(V,E)$ is a hypergraph with a cycle decomposition ${\cal C}$. Sequentially concatenating pairs of cycles in ${\cal C}$ with a common anchor, until no such pairs remain, yields an Euler family for $H$.
\end{proof}

In the remainder of this section, we shall focus on the relationship between eulerian properties and existence of 2-factors in a hypergraph. In Section~\ref{sec:L(H)} we observed that an Euler family in a hypergraph corresponds to a 2-factor in the intersection graph, and an Euler tour corresponds to a Hamilton cycle, but not conversely. As we shall see below, a stronger relationship exists between eulerian properties of a hypergraph and existence of 2-factors in its dual.

Recall that the {\em dual} of a non-empty hypergraph $H=(V,E)$ is the hypergraph $H^T=(E,V^T)$, where $V^T=\{v^T:v \in V \}$ and $v^T=\{ e \in E: v \in e\}$ for all $v \in V$. Observe that $(v,e) \in F(H)$ if and only if $(e,v^T) \in F(H^T)$, whence $(H^T)^T=H$ and the incidence graphs of $H$ and $H^T$ are isomorphic.

It is easy to see that a hypergraph is 2-regular if and only if its dual is 2-uniform. Below, we extend this observation to existence of 2-factors.

\begin{lemma}\label{lem:2-factor}
Let $H=(V,E)$ be a non-empty hypergraph and $H^T=(E,V^T)$ its dual. Let $E' \subseteq E$, $H'=(V,E')$, and $G'=\G(H')$.  Then the following are equivalent:
\begin{enumerate}
\item $H'$ is a 2-factor of $H$.
\item $G'$ satisfies $\deg_{G'}(v)=2$ for all $v \in V$, $\deg_{G'}(e)=|e|$ for all $e \in E'$, and $V(G')\cap E=E'$.
\item $H^T[E']$ is 2-uniform with $|V|$ edges; that is, each edge of $H^T$ intersects $E'$ in exactly 2 vertices.
\end{enumerate}
\end{lemma}

\begin{proof}
It is clear that Statements (1) and (2) are equivalent.

Let $V'=\bigcup_{e \in E'} e$, and recall that $H[E']=(V',E')$. To see that (1) and (3) are equivalent, first observe that $\G(H[E'])$ and $\G(H^T[E'])$ are isomorphic with the isomorphism $\phi: V(\G(H[E'])) \rightarrow V(\G(H^T[E']))$ defined by $\phi(v)=v^T \cap E'$ for all $v \in V'$ and $\phi(e)=e$ for all $e \in E'$. If $(V,E')$ is a 2-factor, then $(V,E')=H[E']$, and since $\G(H[E'])$ and $\G(H^T[E'])$ are isomorphic and $H[E']$ is 2-regular with $|V|$ vertices, $H^T[E']$ is 2-uniform with $|V|$ edges. Conversely, if $H^T[E']$ is 2-uniform with $|V|$ edges, then $H[E']$ is 2-regular with $|V|$ vertices. Hence $H[E']=(V,E')$ and this subhypergraph is a 2-factor of $H$.
\end{proof}

The above lemma easily implies the following.

\begin{cor}
Let $H=(V,E)$ be a non-empty hypergraph and $H^T=(E,V^T)$ its dual. Then $H$ admits a 2-factorization if and only if there exists a partition $\{ E_1,\ldots, E_k\}$ of $E$  such that for all $i\in \{1,\ldots,k\}$, the vertex-induced subgraph $H^T[E_i]$ is 2-uniform with $|V|$ edges (that is, each edge of $H^T$ intersects $E_i$ in exactly 2 vertices).
\end{cor}

We are now ready to demonstrate the correspondence between 2-factors in a hypergraph with no odd-size edges and particular Euler families in its dual.

\begin{theo}\label{the:2-factor}
Let $H=(V,E)$ be a non-empty hypergraph without empty edges
such that $|e|$ is even for all $e \in E$, and let $H^T=(E,V^T)$ be its dual. Let $E' \subseteq E$.

Then $(V,E')$ is a 2-factor (connected 2-factor) of $H$ if and only if $H^T$ has an Euler family (Euler tour, respectively) with anchor set $E'$ that traverses every vertex $e \in E'$ exactly $\frac{|e|}{2}$ times.
\end{theo}

\begin{proof}
Observe that, with the assumptions of the theorem, the dual $H^T$ has no vertices of odd degree.

Assume $F=(V,E')$ is a 2-factor of $H$, and let $G'$ be its incidence graph. By Lemma~\ref{lem:2-factor}, $G'$ is a subgraph of the incidence graph $\G(H)$ such that $\deg_{G'}(v)=2$ for all $v \in V$, $\deg_{G'}(e)=|e|$ for all $e \in E'$, and $V(G')\cap E=E'$. Hence the incidence graph of the dual $H^T$ admits a subgraph that is 2-regular on $V^T$ and even on $E$, which implies that $H^T$ admits an Euler family. Since $\deg_{G'}(e)=|e|$ for all $e \in E'$, this Euler family of $H^T$ traverses each $e \in E'$ exactly $\frac{|e|}{2}$ times, and each $e \in E-E'$ not at all. If, in addition, the 2-factor $F$ is connected, then $G'$ is connected by \cite[Theorem 3.11]{BahSaj}, and hence corresponds to an Euler tour of $H^T$.

The converse is proved by reversing the above steps.
\end{proof}

\begin{cor}
Let $H=(V,E)$ be a non-empty hypergraph without empty edges such that $|e|$ is even for all $e \in E$, and let $H^T=(E,V^T)$ be its dual.

Then $H$ admits a 2-factorization  if and only if there exists a partition $\{ E_1,\ldots, E_k\}$ of $E$ such that for each $i=1,\ldots,k$, the dual $H^T$ admits an Euler family with anchor set $E_i$ that traverses every vertex $e \in E_i$ exactly $\frac{|e|}{2}$ times..
\end{cor}

\begin{proof}
Let $\{F_1,\ldots,F_k\}$ be a 2-factorization of $H$. For each $i \in \{ 1,\ldots,k\}$, let $E_i=E(F_i)$. Then $\{ E_1,\ldots,E_k\}$ is a partition of $E$, and by Theorem~\ref{the:2-factor}, for each $i \in \{ 1,\ldots,k\}$,
the dual $H^T$ admits an Euler family with anchor set $E_i$ that traverses every vertex $e \in E_i$ exactly $\frac{|e|}{2}$ times.

Conversely, assume that $\{ E_1,\ldots,E_k\}$ is a partition of $E$ such that, for each $i \in \{ 1,\ldots,k\}$,
the dual $H^T$ admits an Euler family with anchor set $E_i$ that traverses every vertex $e \in E_i$ exactly $\frac{|e|}{2}$ times. Then by Theorem~\ref{the:2-factor}, each  $F_i=(V,E_i)$ is a 2-factor of $H$, and $\{F_1,\ldots,F_k\}$ is a 2-factorization.
\end{proof}

\begin{center}
{\large \bf Acknowledgement}
\end{center}

The first author wishes to thank the Department of Mathematics and Statistics, University of Ottawa, for its hospitality during his postdoctoral fellowship, when this research was conducted. The second author gratefully acknowledges financial support by the Natural Sciences and Engineering Research Council of Canada (NSERC).


\medskip


\begin{thebibliography}{99}
\bibitem{AlsHeiMoh} Brian Alspach, Katherine Heinrich, Bojan Mohar,  A note on Hamilton cycles in block-intersection graphs, {\em Finite geometries and combinatorial designs}, Contemp. Math. {\bf 111} (1990), 1--4.
\bibitem{AnsJAlg85}  R. P. Anstee, An algorithmic proof of {T}utte's {$f$}-factor theorem, {\em J. Algorithms} {\bf 6} (1985), 112--131.
\bibitem{BahSaj} M. Amin Bahmanian, Mateja \v{S}ajna, Connection and separation in hypergraphs, {\em Theory and Applications of Graphs}, {\bf 2} (2015), Iss. 2, Article 5.
\bibitem{BonMur} J. A. Bondy, U. S. R. Murty, {\em Graph theory}. Graduate Texts in Mathematics {\bf 244}, Springer, New York, 2008.
graphs, {\em J. Combin. Theory B} {\bf 45} (1988), 185--198.
\bibitem{Dew} Megan Dewar, Brett Stevens, {\em Ordering block designs. Gray codes, universal cycles and configuration orderings}, CMS Books in Mathematics, Springer, New York, 2012.
\bibitem{Edm} Jack Edmonds, Paths, trees, and flowers, {\em Canad. J. Math.} {\bf  17} (1965), 449--467.
\bibitem{Eul} Leonhard Euler, Solutio problematis ad geometriam situs pertinentis, {\em Comment. Academiae Sci. Petropolitanae} {\bf 8} (1741), 128--140
\bibitem{Fle} Herbert Fleischner, Spanning Eulerian subgraphs, the splitting lemma, and Petersen's theorem, {\em Discrete Math.} {\bf  101} (1992), 33--37.
\bibitem{GarJohTar} M. R. Garey, D. S. Johnson, R. Endre Tarjan,
The planar Hamiltonian circuit problem is NP-complete,
{\em SIAM J. Comput.} {\bf 5} (1976),  704--714.
\bibitem{Hal} P. Hall, On representatives of subsets, {\em J. London Math. Soc.}, {\bf 10}, 26--30.
\bibitem{HieWie} C. Hierholzer, C. Wiener, \"{U}ber die M\"{o}glichkeit, einen Linienzug ohne Wiederholung und ohne Unterbrechung zu umfahren, {\em  Math. Ann.} {\bf 6} (1873), 30--32.
\bibitem{Hil} A. J. W. Hilton, Canonical edge-colourings of locally finite graphs, {\em Combinatorica} {\bf 2} (1982), 37--51.
\bibitem{HorHur} Victoria Horan, Glenn Hurlbert, 1-overlap cycles for Steiner triple systems, {\em Des. Codes Cryptogr.} {\bf 72} (2014), 637--651.
\bibitem{HorHur4} Victoria Horan, Glenn Hurlbert, Overlap cycles for Steiner quadruple systems, {\em J. Combin. Des.} {\bf 22} (2014),  53--70.
\bibitem{LonNar} Zbigniew Lonc, Pawe{\l} Naroski,  On tours that contain all edges of a hypergraph, {\em Electron. J. Combin.} {\bf 17} (2010), no. 1, Research Paper 144, 31 pp.
\bibitem{Lov} L. Lov\'{a}sz, The factorization of graphs II, {\em Acta Math. Acad. Sci. Hungar.} {\bf 23} (1972), 223--246.
\bibitem{SajWag} Mateja \v{S}ajna, Andrew Wagner, Triple systems are eulerian, submitted.
\bibitem{SamPas} E. Sampathkumar, L. Pushpalatha,  Eulerian hypergraphs {\em Adv. Stud. Contemp. Math.} {\bf 8} (2004), 115--119.
\bibitem{Sta} Nicholas Stamplecoskie, {\em Eulerian Properties of Hypergraphs: Open Trails}, Undergraduate Research Project, University of Ottawa, 2016.
\bibitem{Veb} Oswald Veblen, An Application of Modular Equations in Analysis Situs, {\em Ann. of Math. (2)}, {\bf 14} (1912/13), 86--94.
\bibitem{Tut} W. T. Tutte, A  short proof of the factor theorem for finite graphs, {\em Canadian J. Math.} {\bf  6} (1954), 347--352.
\end{thebibliography}
\end{document}